\documentclass[11pt,reqno,tbtags]{amsart}

\title[Asymptotic normality of fringe subtrees]
{Asymptotic normality of fringe subtrees
and additive functionals in conditioned Galton--Watson trees}
\date{4 December, 2013}

\newcommand\urladdrx[1]{{\urladdr{\def~{{\tiny$\sim$}}#1}}}

\author{Svante Janson}
\thanks{Partly supported by the Knut and Alice Wallenberg Foundation}
\address{Department of Mathematics, Uppsala University, PO Box 480,
SE-751~06 Uppsala, Sweden}
\email{svante.janson@math.uu.se}
\urladdrx{http://www2.math.uu.se/~svante/}

\usepackage{amssymb}

\usepackage{url}
\usepackage[square,numbers]{natbib}

\numberwithin{equation}{section}

\renewcommand\le{\leqslant}
\renewcommand\ge{\geqslant}

\allowdisplaybreaks

\newtheorem{theorem}{Theorem}[section]
\newtheorem{lemma}[theorem]{Lemma}

\newtheorem{corollary}[theorem]{Corollary}

\theoremstyle{definition}
\newtheorem{example}[theorem]{Example}

\newtheorem{remark}[theorem]{Remark}

\theoremstyle{remark}

\newenvironment{romenumerate}[1][0pt]{
\addtolength{\leftmargini}{#1}\begin{enumerate}
 }{\end{enumerate}}

\newcounter{oldenumi}
{\setcounter{oldenumi}{\value{enumi}}
\begin{romenumerate} \setcounter{enumi}{\value{oldenumi}}}
{\end{romenumerate}}

\newcounter{thmenumerate}

\newcounter{romxenumerate}   

\newcounter{xenumerate}   
\newenvironment{xenumerate}
{\begin{list}
    {\upshape(\roman{xenumerate})}
    {\setlength{\leftmargin}{0pt}
     \setlength{\rightmargin}{0pt}
     \setlength{\labelwidth}{0pt}
     \setlength{\itemindent}{\labelsep}
     \setlength{\topsep}{0pt}
     \usecounter{xenumerate}} }
  {\end{list}}

\newcommand\pfitem[1]{\par(#1):}

\newcommand\pfcase[2]{\smallskip\noindent\emph{Case #1\textup: #2} }
\newcommand\pfcasexx[1]{\smallskip\noindent\emph{#1\textup:}}
\newcounter{jeppe}
\newcommand\pfcasex[2]{\stepcounter{jeppe}%
  \pfcase{#1\textup{(\roman{jeppe})}}{#2}}

\newcommand{\refT}[1]{Theorem~\ref{#1}}
\newcommand{\refC}[1]{Corollary~\ref{#1}}
\newcommand{\refL}[1]{Lemma~\ref{#1}}
\newcommand{\refR}[1]{Remark~\ref{#1}}
\newcommand{\refS}[1]{Section~\ref{#1}}

\newcommand{\refE}[1]{Example~\ref{#1}}

\begingroup
  \divide\count255 by 60
  \multiply\count255 by -60
  \advance\count255 by \time
  \ifnum \count255 < 10 \xdef\klockan{\the\count1.0\the\count255}
  \else\xdef\klockan{\the\count1.\the\count255}\fi
\endgroup

\newcommand\nopf{\qed}   

\DeclareMathOperator*{\sumy}{\sum\nolimits^{\prime}}

\newcommand{\sumki}{\sum_{k=1}^\infty}

\newcommand{\sumni}{\sum_{n=1}^\infty}
\newcommand{\sumin}{\sum_{i=1}^n}
\newcommand{\sumim}{\sum_{i=1}^m}

\newcommand{\sumkn}{\sum_{k=1}^n}

\newcommand\set[1]{\ensuremath{\{#1\}}}

\newcommand\xpar[1]{(#1)}
\newcommand\bigpar[1]{\bigl(#1\bigr)}
\newcommand\Bigpar[1]{\Bigl(#1\Bigr)}
\newcommand\biggpar[1]{\biggl(#1\biggr)}
\newcommand\lrpar[1]{\left(#1\right)}

\newcommand\xcpar[1]{\{#1\}}

\newcommand\bigabs[1]{\bigl|#1\bigr|}
\newcommand\Bigabs[1]{\Bigl|#1\Bigr|}

\newcommand\lrabs[1]{\left|#1\right|}
\def\rompar(#1){\textup(#1\textup)}    
\newcommand\xfrac[2]{#1/#2}

\newcommand\parfrac[2]{\lrpar{\frac{#1}{#2}}}

\newcommand\Bigparfrac[2]{\Bigpar{\frac{#1}{#2}}}
\newcommand\biggparfrac[2]{\biggpar{\frac{#1}{#2}}}

\def\xexp(#1){e^{#1}}

\newcommand\floor[1]{\lfloor#1\rfloor}

\newcommand\setn{\set{1,\dots,n}}

\newcommand\ntoo{\ensuremath{{n\to\infty}}}
\newcommand\Ntoo{\ensuremath{{N\to\infty}}}

\newcommand\ktoo{\ensuremath{{k\to\infty}}}

\newcommand\punkt[1]{\if.#1\else.\spacefactor1000\fi{#1}}
\newcommand\iid{i.i.d\punkt}    
\newcommand\ie{i.e\punkt}
\newcommand\eg{e.g\punkt}

\newcommand\cf{cf\punkt}
\newcommand{\as}{a.s\punkt}

\newcommand\ii{\mathrm{i}}

\newcommand{\tend}{\longrightarrow}
\newcommand\dto{\overset{\mathrm{d}}{\tend}}
\newcommand\pto{\overset{\mathrm{p}}{\tend}}

\newcommand\eqd{\overset{\mathrm{d}}{=}}

\newcommand\bbR{\mathbb R}

\newcommand\bbN{\mathbb N}

\newcommand\bbZ{\mathbb Z}

\newcounter{CC}
\newcommand{\CC}{\stepcounter{CC}\CCx} 
\newcommand{\CCx}{C_{\arabic{CC}}}     
\newcommand{\CCreset}{\setcounter{CC}0} 
\newcounter{cc}
\newcommand{\cc}{\stepcounter{cc}\ccx} 
\newcommand{\ccx}{c_{\arabic{cc}}}     
\newcommand{\ccdef}[1]{\xdef#1{\ccx}}     
\newcommand{\ccreset}{\setcounter{cc}0} 
\newcounter{BB}
\newcommand{\BB}{\stepcounter{BB}\BBx} 
\newcommand{\BBx}{B_{\arabic{BB}}}     
\newcommand{\BBdef}[1]{\xdef#1{\BBx}}     

\newcommand\E{\operatorname{\mathbb E{}}}
\newcommand\PP{\operatorname{\mathbb P{}}}
\newcommand\Var{\operatorname{Var}}
\newcommand\Cov{\operatorname{Cov}}

\newcommand\Po{\operatorname{Po}}
\newcommand\Bi{\operatorname{Bi}}
\newcommand\Bin{\operatorname{Bin}}

\newcommand\Ge{\operatorname{Ge}}

\newcommand\sign{\operatorname{sign}}

\newcommand\ga{\alpha}
\newcommand\gb{\beta}
\newcommand\gd{\delta}
\newcommand\gD{\Delta}
\newcommand\gf{\varphi}
\newcommand\gam{\gamma}
\newcommand\gG{\Gamma}
\newcommand\gk{\varkappa}

\newcommand\gs{\sigma}
\newcommand\gss{\sigma^2}
\newcommand\eps{\varepsilon}

\renewcommand\phi{\xxx}  

\newcommand\cF{\mathcal F}

\newcommand\cL{{\mathcal L}}

\newcommand\cT{{\mathcal T}}

\newcommand\ett[1]{\boldsymbol1\xcpar{#1}}

\newcommand\qw{^{-1}}
\newcommand\qww{^{-2}}
\newcommand\qq{^{1/2}}
\newcommand\qqw{^{-1/2}}
\newcommand\qqc{^{3/2}}
\newcommand\qqcw{^{-3/2}}

\renewcommand{\=}{:=}

\newcommand\intoo{\int_0^\infty}
\newcommand\intoooo{\int_{-\infty}^\infty}

\newcommand\dd{\,\mathrm{d}}

\newcommand{\chf}{characteristic function}

\newcommand\lhs{left-hand side}
\newcommand\rhs{right-hand side}

\newcommand\Tx{T_*}
\newcommand\st{\mathfrak T}
\newcommand\stn{\st_n}

\newcommand\sgrt{simply generated random tree}

\newcommand\GW{Galton--Watson}
\newcommand\GWt{\GW{} tree}
\newcommand\cGWt{conditioned \GW{} tree}
\newcommand\GWp{\GW{} process}

\newcommand\ctnx{\cT_{n,*}}

\newcommand\ct{\cT}
\newcommand\ctn{\cT_n}
\newcommand\ctk{\cT_k}
\newcommand\hct{\hat\cT}
\newcommand\ddn{(d_1,\dots,d_n)}
\newcommand\ddl{(d_1,\dots,d_\ell)}

\newcommand\ddil{(d_i,\dots,d_{i+k-1})}

\newcommand\ds{degree sequence}
\newcommand\xixx[1]{\xi_1,\dots,\xi_{#1}}
\newcommand\xix[1]{(\xixx{#1})}
\newcommand\xik{\xix k}

\newcommand\xikx{\xixx k}

\newcommand\xijmx{\xi_{j+1},\dots,\xi_{j+m}}
\newcommand\xijm{(\xijmx)}

\newcommand\tgf{\tilde\varphi}
\newcommand\txi{\tilde\xi}
\newcommand\intpipi{\frac1{2\pi}\int_{-\pi}^{\pi}}
\newcommand\intpipix{\int_{-\pi}^{\pi}}

\newcommand\subsectionx{\subsection{\!}}

\newcommand\ntt{n_{T'}(T)}
\newcommand\gbx{\gb^*}
\newcommand\NN{^{(N)}}
\newcommand\MM{^{(M)}}
\newcommand\MMi{^{(M-1)}}
\newcommand\au{s}
\newcommand\bv{b}
\newcommand\tY{\tilde Y}
\newcommand\gamN{(\gam\NN)}
\newcommand\gamNx{\gam\NN}
\newcommand\ptmt{\pi\MM_T}
\newcommand\hptmt{\hat\pi\MM_T}
\newcommand\tg{\tilde g}
\newcommand\xas{\text{\quad a.s.}}

\newcommand\CS{Cauchy--Schwarz}
\newcommand\CSineq{\CS{} inequality}

\begin{document}




\begin{abstract} 
We consider conditioned Galton--Watson trees and show asymptotic normality
of additive functionals that are defined by toll functions that are not too
large. This includes, as a special case, asymptotic normality of the number
of fringe subtrees isomorphic to any given tree, and joint asymptotic
normality for several such subtree counts. Another example is the number of
protected nodes. The offspring distribution defining the random tree
is assumed to have expectation 1 and finite variance; no further moment
condition is assumed.
\end{abstract}

\maketitle

\section{Introduction}\label{S:intro}
All trees in this paper are \emph{rooted} and \emph{ordered} (= \emph{plane}).
(We assume that the trees are ordered, \ie, that the children of each node
are ordered, for technical convenience. In applications, the ordering is
often irrelevant, and we may then treat unordered trees too by using a
random labelling.) 
We consider in the present paper only finite trees 
(except $\hct$ in \refL{Lfloc}), and  
denote the size (or order), \ie{} the number of nodes, of a tree $T$ by $|T|$.
We let $\st$ denote the countable set of all   
ordered rooted trees (where we identify trees that are isomorphic
in the natural way, with an isomorphism preserving the root and the
orderings of children); 
let further
$\stn$ be the (finite) subset of all such trees of order $n$.   
(See further \eg{} \cite{Drmota} and \cite{SJ264}.)

Given a rooted tree $T$ and a node $v$ in $T$, let $T_v$ be the subtree of
$T$ rooted at $v$, \ie, the subtree consisting of $v$ and all its
descendents.
Such subtrees are called \emph{fringe subtrees}.
(By ``subtree'', we mean in the present paper always
a fringe subtree, except in \refE{EWagner}.)
We are interested in the collection \set{T_v} of all fringe subtrees of a given
tree $T$. 

One way to study this collection is to  consider
the  \emph{random fringe subtree}  $\Tx$,
which is the random rooted tree obtained by taking the subtree $T_v$ 
at a uniformly random node $v$ in $T$. This was
introduced and studied by \citet{Aldous-fringe}, both in general and for many
important examples.
We let, for $T,T'\in\st$, 
\begin{equation}\label{ntt}
  n_{T'}(T):=\bigabs{\set{v\in T':T_v=T'}},
\end{equation}
\ie, the number 
of subtrees of $T$ that are equal 
 (\ie,  isomorphic to)
to $T'$. Then the distribution of $\Tx$ is
given by
\begin{equation}\label{ptx}
  \PP(\Tx=T')=n_{T'}(T)/|T|, \qquad T'\in\st.
\end{equation}
Thus,  to study the distribution of $\Tx$ 
is equivalent to studying the numbers
$n_{T'}(T)$. 

A related point of view is to 
let $f$ be a functional of rooted trees, \ie,  a function $f:\st\to\bbR$,
and for a tree $T\in\st$ 
consider the sum
\begin{equation}\label{F}
  F(T)=F(T;f)\=\sum_{v\in T} f(T_v).
\end{equation}
Thus,
\begin{equation}\label{FE}
  F(T)/|T|=\E f(\Tx).
\end{equation}
One important example of this is to take $f(T)=\ett{T=T'}$, the indicator
function that $T$ equals
some given tree $T'\in\st$; then $F(T)=n_{T'}(T)$ and \eqref{FE} reduces to
\eqref{ptx}. 
Conversely, for any $f$,
\begin{equation}\label{Fnt}
  F(T)=\sum_{T'\in\st}f(T') n_{T'}(T);
\end{equation}
hence any $F(T)$ can be written as a linear combination of the 
subtree counts $\ntt$, so the two points of views are essentially equivalent.

\begin{remark}
Functionals $F$ that can be written as \eqref{F} for some $f$ are called
\emph{additive functionals}.
  The definition \eqref{F} can also be written recusively as
  \begin{equation}\label{toll}
	F(T)=f(T)+\sum_{i=1}^d F(T_i),
  \end{equation}
where $T_1,\dots,T_d$ are the branches 
(\ie, the subtrees rooted at the children of the root) 
of $T$.
In this context, $f(T)$ is often called a \emph{toll function}.
(One often considers toll functions that depend only on the size $|T|$ of
$T$, but that is 
not always the case. We emphasise that we allow more general functionals $f$.)
\end{remark}

Note that 
when $T$ is a random tree, as it was in \cite{Aldous-fringe} 
and will be in the present paper,
$F(T)$ is a random variable.
In particular,  $\ntt$ is a random variable for each $T'\in\st$, and thus
the distribution of $\Tx$, which is given by \eqref{ptx}, is a random
probability distribution on $\st$. Note that 
\eqref{ptx} now reads 
\begin{equation}\label{ptxx}
  \PP\bigpar{\Tx=T'\mid T}=n_{T'}(T)/|T|
\end{equation}
and that similarly
\eqref{FE} then has to be
replaced by 
\begin{equation}\label{FEx}
  F(T)/|T|=\E \bigpar{f(\Tx)\mid T}.
\end{equation}

\begin{remark}
This is the  \emph{quenched} version of the fringe subtree
$\Tx$, where
we first select a realization of the random tree $T$, and then
fix this realization and choose $v\in T$ uniformly at
random, yielding a fringe subtree $\Tx$ with a distribution depending on
$T$;
this is thus a random distribution, as said above.
The alternative is 
the \emph{annealed} version where we take a random tree $T$ and a uniformly
random node $v$ in it as a  combined random event; this yields a
random fringe 
subtree with a distribution that is the expectation of the random
distribution \eqref{ptx} in the quenched version. 
When $|T|$ is fixed (as in the cases we study in the present paper),
the annealed version thus corresponds
to considering only the expectation $\E F(T)=|T|\E f(\Tx)$ of the sum \eqref{F},
or equivalently $\E f(\Tx)$,
while the quenched version corresponds to studying the conditional
expectation \eqref{FEx}.
\end{remark}

The random trees that we consider in this paper are 
\cGWt{s}, see \refS{Strees} for definition and notation.
(Related results for some other random trees are given by
\citet{FillKapur-mary,FillKapur-mary2} 
($m$-ary search trees under different models)
and
\citet{HJ} (random binary search trees and random recursive trees).)
The \GW{} trees are defined using an offspring distribution; we let $\xi$ denote
a random variable with this distribution and we assume throughout the paper
that the mean $\E\xi=1$ and (except in \refT{T0}) 
that the variance $\gss:=\Var\xi$ is finite (and non-zero).
We recall the well-known fact that several standard examples of random trees
can de defined in this way, for example
uniform random ordered trees ($\xi\sim\Ge(1/2)$, $\gss=2$),
uniform random labelled trees ($\xi\sim\Po(1)$, $\gss=1$) and
uniform random binary trees ($\xi\sim\Bi(2,1/2)$, $\gss=1/2$), 
see \eg{} \citet{AldousII}, \citet{Devroye}, \citet{Drmota}, \citet{SJ264}.

The results in \citet{Aldous-fringe} focus on convergence (in probability),
as $|T|\to\infty$, 
of the fringe subtree distribution for suitable classes of random trees $T$,
which by \eqref{FEx} is equivalent to convergence 
of $F(T)/|T|$ 
or $\E F(T)/|T|$
for suitable functionals $f$.
For the
\cGWt{s} studied here, this is stated in the 
following theorem. 
Part (i) was proved by \citet{Aldous-fringe}, assuming $\Var\xi<\infty$ as 
we assume in the rest of the paper, and extended to more general $\xi$ by 
\citet{BenniesK}, and further by \citet{SJ264};
the sharper version (ii) is proved in \cite[Theorem 7.12]{SJ264}.

\begin{theorem}[Aldous, et al.]\label{T0}
Let $\ctn$ be a \cGWt{} with $n$ nodes, defined by an offspring distribution
$\xi$ with $\E\xi=1$,
and let $\cT$ be the corresponding unconditioned \GWt.
Then,  as \ntoo:
  \begin{romenumerate} [-10pt]
  \item ({Annealed version}.)
The fringe subtree $\ctnx$ converges in distribution to the \GWt{} $\cT$.
I.e., for every fixed tree $T$, 
\begin{equation}\label{t2a}
\frac{\E n_T(\ctn)}n=  \PP(\ctnx=T) \to \PP(\cT=T).
\end{equation}
Equivalently, for any bounded functional $f$ on $\st$,
\begin{equation}\label{t2af}
\E \frac{F(\ctn)}{n} = \E f(\ctnx)\to \E f(\cT).
\end{equation}

  \item ({Quenched version}.)
The conditional distributions $\cL(\ctnx\mid\ctn)$ converge to the
distribution of $\cT$ in probability.
I.e., for every fixed tree $T$, 
\begin{equation}\label{t2q}
\frac{n_T(\ctn)}n=  
  \PP(\ctnx=T\mid\ctn) \pto \PP(\cT=T).
\end{equation}
Equivalently, for any bounded functional $f$ on $\st$,
\begin{equation}\label{t2qf}
  \frac{F(\ctn)}{n} = \E f\bigpar{\ctnx\mid \ctn}\pto \E f(\cT).
\end{equation}
\end{romenumerate} 
\qed
\end{theorem}

\begin{remark}
  The statement in \cite[Theorem 7.12]{SJ264} uses \eqref{t2a} and
  \eqref{t2q}, here expanded using  \eqref{ptxx}.
The equivalences with \eqref{t2af} and \eqref{t2qf} follow by standard
properties of convergence in distribution, see \eg{} 
\cite[Theorem 2.1 and Section 4]{Billingsley}.
(Note that the set of finite ordered trees is a
countable discrete set, which simplifies the situation and \eg{} justifies
that it is enough to consider point probabilities in \eqref{t2a} and
\eqref{t2q}. To show \eqref{t2qf} it may be convenient to use 
the Skorohod representation theorem \cite[Theorem 4.30]{Kallenberg} and
assume that \eqref{t2q} holds \as{} for every $T$.)

The result is easily extended to include also unbounded $f$ with suitable
growth conditions, see for example \refT{T1}\ref{T1e},\ref{T1v} 
and \refR{Rp} below.
\end{remark}

\refT{T0} is a law of large numbers for $F(\ctn)$.
In the present paper we take the next step and
study the fluctuations of $F(\ctn)$; we prove a central
limit theorem, \ie, 
asymptotic normality of $F(\ctn)$ under suitable assumptions. 
This includes, as a special case, (joint) normal convergence of the subgraph
counts $\ntt$, see \refC{C1}.
Our main result is the following.
(The proof of this and the following results is given in \refS{Spf}.)

\begin{theorem}
  \label{T1}
Let $\ctn$ be a \cGWt{} of order $n$ with offspring distribution $\xi$, 
where $\E\xi=1$ and $0<\gss\=\Var\xi <\infty$, and let $\cT$ be the
corresponding unconditioned \GWt.
Suppose that $f:\st\to\bbR$ is a functional of rooted trees
such that $\E|f(\cT)|<\infty$,
and let 
$\mu:=\E f(\cT)$.
\begin{romenumerate}[-10pt]
\item \label{T1e}
If\/ $\E f(\ctn)\to0$ as \ntoo, 
then
\begin{equation}\label{t1e}
\E F(\ctn) 
=n\mu+o\bigpar{\sqrt n}.
\end{equation}

\item \label{T1v}
If  
\begin{align}
\E f(\ctn)^2&\to0 \label{t1v1}
\intertext{as \ntoo, and}
\label{t1v2}
\sumni \frac{\sqrt{\E (f(\ctn)^2)}}{n} &<\infty,	
  \end{align}
then
\begin{equation}\label{t1v}
  \Var F(\ctn) = n\gam^2+o(n)
\end{equation}
where 
\begin{equation}\label{gam}
  \gam^2 := 2\E\Bigpar{f(\cT)\bigpar{F(\cT)-|\cT|\mu}}
-\Var f(\cT) -\mu^2/\gss
\end{equation}
is finite; 
moreover,
\begin{equation}\label{t1n}
\frac{ F(\ctn) - n\mu}{\sqrt n} \dto N(0,\gam^2).
\end{equation}
\end{romenumerate}
\end{theorem}

By \eqref{t1e}, we may replace $n\mu$ by the exact mean $\E F(\ctn)$
in \eqref{t1n}.

\begin{remark}\label{RT1a}  
 By \eqref{pett}, the condition $\E|f(\cT)|<\infty$ is equivalent to
$\sum_n n\qqcw\E|f(\ctn)|<\infty$; in particular, this holds if
$\E|f(\ctn)|=O(1)$, and thus if \eqref{t1v1} holds.
(It is also implied by \eqref{t1v2}.)
\end{remark}

\begin{remark}\label{Rgam0}
  It follows from \eqref{t1v} that $\gam^2\ge0$. We do not know whether
  $\gam^2=0$ is possible except in trivial cases when $F(\ctn)$ is
  deterministic for all $n$.
\end{remark}

Special cases of \refT{T1} have been proved before, by various methods.
A simple example is the number of leaves in $\ctn$, shown to be normal by
\citet{Kolchin}, see \refE{Eleaves}. 
(See also \citet[Remark 7.5.3]{Aldous-fringe}.)
\citet{Wagner} considered random labelled trees (the case $\xi\sim\Po(1)$)
and showed \refT{T1} (and convergence of all moments)
for this case, assuming further that $f$ is bounded and $\E|f(\ctn)|=O(c^n)$
for some $c<1$ (a stronger assumption that our \eqref{t1v1}--\eqref{t1v2}).

\refT{T1} is stated for a single functional $F$, but joint convergence for
several different $F$ (each satisfying the conditions in the theorem)
follows immediately by the Cram\'er--Wold device (\ie, by considering linear
combinations); the asymptotic covariances follow from \eqref{gam} by
polarization in the usual way (\ie, using \eg{}
$\Cov(X,Y)=\frac14(\Var(X+Y)-\Var(X-Y))$).
One example is the following corollary for the subtree counts \eqref{ntt};
by \eqref{ptxx}, this corollary shows that the 
conditional distribution $\cL(\ctnx\mid\ctn)$
of the fringe subtree $\ctnx$ of $\ctn$ has asymptotically Gaussian
fluctuations around the limit distribution given by \refT{T0}.

\begin{corollary}\label{C1}
 The subtree counts $n_T(\ctn)$, $T\in\st$, are asymptotically
jointly normal. 
More precisely, let $\pi_T:=\PP(\cT=T)$,
 \begin{align}
  \gam_{T,T}&:=
\pi_{T}
 -\bigpar{2|T|-1+\gs\qww}\pi_{T}^2, \label{gamT}
\end{align}
and, for $T_1\neq T_2$,
\begin{align}\label{gamTT}
  \gam_{T_1,T_2}&:=
 n_{T_2}(T_1)\pi_{T_1}+ n_{T_1}(T_2) \pi_{T_2}
 -\bigpar{|T_1|+|T_2|-1+\gs\qww}\pi_{T_1}\pi_{T_2}.
 \end{align}
Then, for any trees $T,T_1,T_2\in\st$,
\begin{align}
  \E n_T(\ctn)& = n\pi_T +o\bigpar{\sqrt n},\label{eT}\\
  \Var n_T(\ctn)& = n\gam_{T,T} +o(n), \label{varT}\\
 \Cov \bigpar{n_{T_1}(\ctn),n_{T_2}(\ctn)}& = n\gam_{T_1,T_2} +o(n),
   \label{covT}\\
\frac{n_T(\ctn)-n\pi_T}{\sqrt n} &\dto Z_T, \label{asnT}
\end{align}
the latter jointly for all $T\in\st$, where $Z_T$ are jointly normal with
mean $\E Z_T=0$ and covariances
$\Cov\bigpar{Z_{T_1},Z_{T_2}}=\gam_{T_1,T_2}$.
\end{corollary}

We say that the functional $f$ has \emph{finite support}
if $f(T)\neq0$ only for finitely many trees $T\in\st$;
equivalently,
there exists a constant $K$ such that $f(T)=0$ unless $|T|\le K$.
Note that a functional with finite support necessarily is bounded.
By \eqref{Fnt}, the additive functionals $F$ that arise from functionals $f$
with finite support are exactly the finite linear combinations of subgraph
counts $n_{T'}(T)$. Hence \refC{C1} is equivalent to asymptotic normality
(with convergence of mean and variance) for $F(\ctn)$ whenever $f$ has
finite support.
The asymptotic variance
$\gam^2=\lim_\ntoo \Var F(\ctn)/n$  is given by \eqref{gam} or, equivalently,
follows from \eqref{gamT}--\eqref{gamTT}.

For functionals with finite support, we can show that 
$\gam^2>0$ except in
trivial cases, \cf{} \refR{Rgam0}.

\begin{theorem}
  \label{Tgam} 
Suppose that $f$ is a functional on $\st$ with finite support, and that
$\gam^2=0$. 
\begin{romenumerate}[-10pt]
\item 
If the support \set{k:p_k>0} of $\xi$ contains at least two positive
integers,
then $f(\ct)=F(\ct)=F(\ctn)=0$ a.s.
\item 
Otherwise, \ie, 
if $\set{k:p_k>0}=\set{0,r}$ for some $r>1$,
then $f(\ct)=a\ett{|\ct|=1}$ for some real $a$ 
and $F(\ct)=a(n-(n-1)/r)$ is deterministic.
\end{romenumerate}
\end{theorem}

Equivalently, in (i), the matrix $(\gam_{T_1,T_2})_{|T_1|,|T_2|\le M}$ 
(where we only consider trees $T_1,T_2$ with $\PP(\ct=T_j)>0$)
is positive definite for every $M$; in (ii)
the submatrix 
$(\gam_{T_1,T_2})_{2\le|T_1|,|T_2|\le M}$  is positive definite.

\begin{remark}\label{RT1}
The condition
\eqref{t1v1} in \refT{T1}\ref{T1v}
is equivalent to
$\E f(\ctn)\to0$ together with
$\Var f(\ctn)\to0$, and  it implies
$\E |f(\ctn)|\to0$ as assumed in \ref{T1e}.
Both this condition and \eqref{t1v2} say that $f(T)$ is (on the average, at
least) decreasing as $|T|\to\infty$, but a rather slow decrease is
sufficient; for 
example, the theorem applies when $f(T)=1/\log^2|T|$ (for $|T|>1$).
If we assume better integrability of $\xi$, we can weaken the condition a
little, see \refR{Rbetter}, 
but not by much. 
In particular, it is \emph{not} enough to assume that $f$ is a
bounded functional. For a trivial example, let $f(T)=1$ for all trees $T$;
then $F(T)=|T|$ so $F(\ctn)=n$ is constant, with mean $n$ and variance 0.
However, the first two terms on the \rhs{} of \eqref{gam} vanish, so
$\gam^2=-\gs\qww<0$, which is absurd for an asymptotic variance,
and \eqref{t1v} and \eqref{t1n} fail.
Nevertheless,  in this trivial counterexample, 
it is only the value of $\gam^2$
that is wrong; \eqref{t1v} and  \eqref{t1n} trivially hold with $\gam^2=0$.
\refE{ESW} is a more complicated counterexample where $f$ is bounded (and
$f(T)\to0$ as $|T|\to\infty$ so \eqref{t1v1} holds) but at least one of
\eqref{t1v} and \eqref{t1n} fails (for any finite $\gam^2$); we conjecture
that both fail in this example.
\refE{EJW} is a related example where \eqref{t1v1} holds but 
$\Var F(\ctn)/n\to\infty$. 
\end{remark}

\begin{remark}
If we go further and allow $f(T)$ that grow with the size $|T|$, we cannot
expect the results to hold.
\citet{FillKapur-catalan} have made an interesting and illustrative study 
(for certain $f$)
of the case of binary trees, which is the case $\xi\sim \Bin(2,1/2)$ of
\cGWt, and presumably typical for other \cGWt{s} as well.
They show that for $f(T)=\log|T|$, $F(\ctn)$ is asymptotically normal, but
with a variance of the order $n\log n$.
And if $f(T)$ increases more rapidly, with $f(T)=|T|^\ga$ for some $\ga>0$,
then the variance is of order $n^{1+2\ga}$, and $F(\ctn)$ has, after
normalization, a non-normal limiting distribution.
We conjecture that similar results hold for general \cGWt{s} and increasing
$f$, but the precise limits presumably depend on the offspring distribution
$\xi$. 

Intuitively,
our conditions are such that the sum \eqref{F} is dominated by the many
small subtrees $T_v$; since different parts of our trees are only weakly
dependent on each other, this makes asymptotic normality plausible. For
a toll function $f$ that grows too rapidly with the size of $T$, the sum
\eqref{F} will on the contrary be dominated by 
large subtrees, which are more strongly dependent,
and then other limit distributions will appear. 
\end{remark}

\begin{remark}
For the $m$-ary search tree ($2\le m\le26$) and random recursive tree
a similar theorem holds, but there $f(T)$ may grow almost as
$|T|\qq$,
see \citet{HwangN} (binary search tree, $f$ depends on $|T|$ only),
\citet{FillKapur-mary} ($m$-ary search tree, $f$ depends on $|T|$ only),
\citet{HJ} (binary search tree and random recursive tree, general $f$).
A reason for this difference is that for a \cGWt,
the limit distribution of the size of the fringe subtree,  
which by \refT{T0} is the distribution 
of $|\cT|$, decays rather slowly, with $\PP(|\cT|=n)\asymp n\qqcw$, 
see \eqref{pett},
while 
the corresponding
limit distribution for fringe subtrees in a binary search tree or random
recursive tree decays somewhat faster, as 
$n^{-2}$, see \citet{Aldous-fringe}. 
Cf.~ also the related results in \citet[Theorem 13 and 14]{FillFK},
showing a similar contrast (but at orders $n\qq$ and $n$) between
uniform binary trees (an example of a \cGWt) and binary search trees 
for the asymptotic expectation of an additive functional.
\end{remark}

The counterexamples in Examples \ref{EJW}--\ref{ESW} are constructed to
have rather large correlations between $f(T_v)$ and $f(T_w)$ for different
subtrees $T_v$ and $T_w$. In typical applications, this is not the case, and
we expect \refT{T1} to hold also for nice functions $f$ that do not quite
satisfy \eqref{t1v1} and \eqref{t1v2}.
A simple example is the number of nodes of outdegree $r$, for some fixed
$r\ge0$ (with $p_r>0$). This equals $F(T)$ if we let 
$f(T):=\ett{\text{the root of $T$ has degree $r$}}$.
In this case, \citet[Theorem 2.3.1]{Kolchin} has proved asymptotic
normality, see further Examples \ref{Eleaves}--\ref{Er}.
We can extend this as follows.

We say that a functional $f(T)$ on $\st$ is \emph{local} 
(with \emph{cut-off} $M$)
if it depends only on
the first $M$ generations of $T$, for some  $M<\infty$, \ie, if we let
$T\MM$ denote $T$ truncated at height $M$, then $f(T)=f(T\MM)$.
More generally, we say that $f$ is \emph{weakly local} 
(with \emph{cut-off} $M$)
if $f(T)$ depends on 
$|T|$ and $T\MM$ for some $M$.

\begin{theorem}
  \label{Tlocal}
Let $\ctn$ be a \cGWt{} as in \refT{T1}.  
Suppose that $f:\st\to\bbR$ is a bounded and local functional.
Then the conclusions \eqref{t1e},  \eqref{t1v} and \eqref{t1n} hold
for some $\gam^2<\infty$.

More generally, the same holds if
$f$ is a bounded and weakly local functional such that 
$\E f(\ctn)\to0$ and $\sum_n|\E f(\ctn)|/n<\infty$.
\end{theorem}

The proof in \refS{Spf} shows also 
that the asymptotic variance $\gam^2$  equals
$\lim_\Ntoo \gamN^2$, where $\gamN^2$ is given by
\eqref{gamN} or, in the case of a bounded local functional,
\eqref{gamN} applied to $f(T)-\E f(\hct)$, with $\hct$ defined in \eqref{cant}.

We give some examples in \refS{Sex}. Sections \ref{Strees}--\ref{Sprel}
contain preliminaries. The expectation $\E F(\ctn)$ is studied in
\refS{Sexp}, and \refT{T1}\ref{T1e} is proved.
\refS{Svar} establishes  bounds for the variance $\Var F(\ctn)$, and 
proves the asymptotic \eqref{t1v}
in the special case of a functional $f$ with finite support.
\refS{San} shows asymptotic normality for functionals $f$ with finite
support.
Finally, in \refS{Spf}, the  variance bounds in \refS{Svar} and a truncation
argument are used to extend the latter results to more general $f$,
completing the proofs of the theorems above.

\begin{remark}
\label{Rlocal}
  In \refT{Tlocal}, $\gam^2$ is not always given by \eqref{gam} because 
$\E\bigpar{f(\cT)\bigpar{F(\cT)-|\cT|\mu}}$ does not necessarily exist (in
  the usual sense, as an absolutely convergent integral), see \refE{Er};
thus we in general take limits using truncations.  
\end{remark}

\begin{remark}\label{Rwiener}
Although the \GWt{} $\ct$ is finite a.s., its expected size
$\E|\ct|=\infty$, as is seen from \eqref{pett} or directly from the
definition.
Since the random fringe tree $\ctnx\dto\ct$ by \refT{T0}, it follows that
$\E|\ctnx|\to\infty$; similarly, \refT{T0}(ii) implies
$\E\bigpar{|\ctnx|\mid\ctn}\pto\infty$. 

In fact, for any tree $T$ with $|T|=n$, 
letting $d(v)$ be the depth of $v$ and
defining a partial order on the nodes of $T$ by  $v\le w$ if $v$ is on the
path from the root to $w$, 
\begin{equation*}
  \E|\Tx|=\frac1n\sum_{v\in T} |T_v|
=\frac1n\sum_{v,w\in T} \ett{v\le w}
=\frac1n\sum_{w\in T}\bigpar{d(w)+1}
=1+\frac1n\sum_{w\in T}{d(w)},
\end{equation*}
\ie, 1 plus the average path length. Well-known results on the average path
length in a \cGWt, see  \citet{AldousII,AldousIII}, thus imply
\begin{equation}
  n\qqw \E\bigpar{|\ctnx|\mid\ctn} \dto \gs\qw\hat\xi,
\end{equation}
where $\hat\xi$ is twice the Brownian excursion area.
Hence, although the distribution of the size of a random fringe tree is
tight, so the size is bounded in probability,
the average fringe tree size is of the order $n\qq$.
Similarly,
\begin{equation*}
  \E|\Tx|^2
=\frac1n\sum_{v,w,u\in T} \ett{v\le w,\,v\le u}
=\frac1n\sum_{u,w\in T}\bigpar{d(w\wedge u)+1}
\end{equation*}
and by \cite[Theorem 3.1]{SJ146},
\begin{equation}
  n^{-3/2} \E\bigpar{|\ctnx|^2\mid\ctn} \dto \gs\qw\eta,
\end{equation}
for a certain positive random variable $\eta$.
Hence the average of the square of the fringe tree size is of 
the order $n^{3/2}$.
\end{remark}

\subsection{Some notation}
All unspecified limits are as \ntoo.

We let $C_1,C_2,\dots$ and $c_1,c_2,\dots$ denote unspecified positive
constants (possibly depending on $f$ and $\xi$, but not on $n$ and other
variables, and possibly different at different occurences).  
(We use $C_i$ for large constants and $c_i$ for small.)
We also use standard $O$ and $o$ notation (with the limit in $o$ 
as \ntoo{} unless otherwise said). Moreover, we sometimes use the less common notation
(for $a_n,b_n\ge0$)
$a_n \ll b_n$ for $a_n=O(b_n)$ (or, equivalently, $a_n\le C_1 b_n$).

The \emph{outdegree} of a node in a tree
is its number of children. (This is, except for the
root, the degree minus 1.) The \emph{degree sequence} of a tree $T\in\stn$ is 
the sequence $(d_1,\dots,d_n)$ of outdegrees of the nodes taken in
depth-first order, \ie, starting with the (out)degree $d_1$ of the root and
then taking the degree 
sequences of the branches $T_{v_i}$ one by one, where $v_1,\dots,v_{d_1}$
are the children of the root, in order.
It is easily seen that 
  a sequence $\ddn\in\bbN^n$ 
(where $\bbN\=\set{0,1,2,\dots}$)
is the degree sequence of a tree $T\in\stn$ if
  and only if 
  \begin{align}\label{ld}
\begin{cases}
\sum_{i=1}^j d_i \ge j,
& 1\le j<n,
\\	
\sum_{i=1}^n d_i =n-1,
\end{cases}
  \end{align}
see \eg{} \cite[Lemma 15.2]{SJ264}.
Note also that a tree in $\st$ is uniquely determined by its
degree sequence.

The \emph{depth} of a node $v$ in a tree 
is its distance to the root; we denote it by $d(v)$.

\section{Examples}\label{Sex} 

\begin{example}\label{Eleaves}
  The perhaps simplest non-trivial example is to take $f(T)=\ett{|T|=1}$.
Then $F(T)$ is the number of leaves in $T$.
We have 
$\E f(\cT)=\PP(|\cT|=1)=\PP(\xi=0)=p_0$.

Theorems \ref{T1} and \ref{Tlocal} both apply and show asymptotic normality
of $F(\ctn)$, 
and so does \refC{C1} since $F(T)=n_\bullet(T)$, where $\bullet$ is the tree
of order 1;
 \eqref{gam} yields
\begin{equation}\label{game0}
  \gam^2 = 2p_0(1-p_0)-p_0(1-p_0)-p_0^2/\gss
=p_0-(1+\gs\qww)p_0^2,
\end{equation}
which also is seen directly from \eqref{gamT}.
The asymptotic normality in this case (and a local limit theorem)
was proved by 
\citet[Theorem  2.3.1]{Kolchin}.
By \refT{Tgam}, or by a simple calculation directly from \eqref{game0},
$\gam^2>0$ except in the case $p_r=1-p_0=1/r$ for some $r\ge2$
when all nodes in $\ctn$ have 0 or $r$ children (full $r$-ary trees)
and $n_\bullet(\ctn)=n-(n-1)/r$ is deterministic.
\end{example}

\begin{example}\label{Er} 
A natural extension is to consider the number of nodes of outdegree $r$, for
some given integer $r\ge1$; we denote this by $n_r(T)$.
Then $n_r(T)=F(T)$ with $f(T)=1$ if the root of $T$ has
degree $r$, and $f(T)=0$ otherwise.
Asymptotic normality of $n_r(\ctn)$ too was proved  by
\citet[Theorem  2.3.1]{Kolchin}, with
\begin{equation} \label{er1}
  n\qqw\bigpar{F(\ctn)-np_r} \dto N\bigpar{0,\gam_r^2}
\end{equation}
where 
\begin{equation}\label{er2}
  \gam_r^2
=p_r(1-p_r)-(r-1)^2p_r^2/\gss,
\end{equation}
see also \citet{SJ132} (joint convergence and moment convergence, assuming
at least $\E\xi^3<\infty$),
\citet{Minami} and \citet[Section 3.2.1]{Drmota} (both assuming an exponential
moment) for different proofs.

It is easily checked that for $r>0$, $\gam_r>0$ except in the two
trivial cases $p_r=0$, when $n_r(\ctn)=0$,
and  $p_r=1-p_0=1/r$,
when all nodes have 0 or $r$ children (full $r$-ary trees)
and $n_r(\ctn)=(n-1)/r$
is deterministic.

In this example,
\begin{equation}
  \E f(\ctn) = \PP(\text{the root of $\ctn$ has degree $r$})
\to rp_r.
\end{equation}
see \cite{Kennedy} and \cite[Theorem 7.10]{SJ264}. 
Hence \eqref{t1v1} and \eqref{t1v2} both
fail, and we cannot apply \refT{T1}. (It does not help to subtract a
constant, since $f(\ctn)$ is an indicator variable.) However, $f$ is a
bounded local functional. Hence \refT{Tlocal} applies and yields
\eqref{er1},
together with convergence of mean and variance, for some $\gam_r$.
It is immediate from the definition of the \GWt{} $\ct$ that 
\begin{equation}\label{ermu}
\mu:=\E f(\cT) = \PP(\text{the root of $\ct$ has degree $r$})
=p_r.
\end{equation}
Similarly, 
we obtain joint convergence for
different $r$ by \refT{Tlocal} and the Cram\'er--Wold device.
(It seems that joint convergence has not been proved before without assuming
at least $\E\xi^3<\infty$.)

Nevertheless, this result is a bit disappointing, since we do not obtain the
explicit formula \eqref{er2} for the variance.
\refT{Tlocal} shows existence of $\gam^2$ but the formula 
(given by the proof)
as a limit of
\eqref{gamN} 
is rather involved, and we do not know any way to derive
\eqref{er2} from it. 
In this example, because of the simple structure of $f$, we can use a
special argument and derive both \eqref{er2} 
and the asymptotic covariance $\gam_{rs}$ for
two different outdegrees $r,s\ge0$:
\begin{equation}\label{erc}
  \gam_{rs} = 
-p_rp_s-(r-1)(s-1)p_rp_s/\gss, \qquad r\neq s,
\end{equation}
(as proved by  \cite{SJ132} provided $\E \xi^3<\infty$);
we give this proof in \refS{Spf}.

Note that by \eqref{er1}, 
$\liminf_\ntoo n\qqw\E|F(\ctn)-n\mu| \ge (2/\pi)\qq\gam_r$,
so assuming $\gam_r>0$, $\E|F(\ctn)-n\mu|\ge \cc n\qq$, at least for large
$n$. It is easily seen that also 
$\E f(\ctn)|F(\ctn)-n\mu|\ge \cc n\qq$, at least for large $n$; hence,
using \eqref{pett},
\begin{equation*}
  \E \bigabs{f(\cT)\bigpar{F(\cT)-|T|\mu}}
=
\sumni \pi_n  \E \bigabs{f(\ctn)\bigpar{F(\ctn)-n\mu}}
=\infty,
\end{equation*}
which shows that the  expectation in \eqref{gam} does not exist, so
$\gam^2$ is not given by \eqref{gam}.
\end{example}

\begin{example}\label{Eprotected}
A node in a (rooted) tree is said to be \emph{protected}
if it is neither a leaf nor the parent of a leaf.
Asymptotics for the expected number of protected nodes 
in various random trees, including 
several  examples of \cGWt{s}, have been given by \eg{} 
\citet{CheonShapiro} and \citet{Mansour}, and 
convergence in probability of
the fraction of protected nodes is proved 
for general \cGWt{s} by \citet{SJ283}.

We can now extend this to asymptotic normality of the number of protected
nodes, in any \cGWt{} $\ctn$ with $\E\xi=1$ and $\gss<\infty$.
We define
$f(T):=\ett{\text{the root of $T$ is protected}}$,
and then $F(T)$ is
the number of protected nodes in $T$.
Since $f$ is a bounded and local functional, \refT{Tlocal} applies and shows
asymptotic normality of $F(\ctn)$.

The asymptotic mean $\mu=\E f(\cT)$ is easily calculated, 
see \cite{SJ283} where also
explicit values are given for several examples of \cGWt{s}.
However, as in \refE{Er}, we do not see how to find an explicit value of
$\gam^2$ from \eqref{gamN}  (although it ought to be
possible to use these for numerical calculation for a specific offspring
distribution). It seems possible that there is some other argument
to find $\gam^2$, perhaps
related to our proof of \eqref{erc} in \refS{Spf}, but we have not pursued this
and we
leave it as an open problem to find the asymptotic variance
$\gam^2$, for example for uniform labelled trees or uniform binary trees.
\end{example}

\begin{example}\label{EWagner}
\citet{Wagner} studied the number $s(T)$
of arbitrary subtrees (not necessarily fringe  subtrees) of the tree $T$,
and the 
number $s_1(T)$ of such subtrees that contain the root.
He noted that if $T$ has branches 
$T_1,\dots,T_d$, then
$s_1(T)=\prod_{i=1}^d (1+s_1(T_i))$ and thus
\begin{equation}
  \log\bigpar{1+s_1(T)}=\log\bigpar{1+s_1(T)\qw}
+\sum_{i=1}^d\log\bigpar{1+s_1(T_i)},
\end{equation}
so $\log\bigpar{1+s_1(T)}$ is an additive functional with toll function 
$f(T)=\log\bigpar{1+s_1(T)\qw}$, see \eqref{toll}.
\citet{Wagner} used this and the special case of \refT{T1} shown by him to show
asymptotic
normality of $\log\bigpar{1+s_1(\ctn)}$ 
(and thus of $\log{s_1(\ctn)}$)
for the case of uniform random
labelled trees (which is $\ctn$ with $\xi\sim\Po(1))$. We can generalize this
to arbitrary \cGWt{s} with $\E\xi=1$ and $\E\xi^2<\infty$ by \refT{T1},
noting that  
$|f(\ctn)|\le s_1(\ctn)\qw\le n\qw$ (since $s_1(T)\ge |T|$ by considering
only paths from the root); hence \eqref{t1v1}--\eqref{t1v2} hold.
Consequently,
\begin{equation}\label{wag}
  \bigpar{\log s_1(\ctn)-n\mu}/\sqrt n \dto N(0,\gam^2)
\end{equation}
for some $\mu=\E \log\bigpar{1+s_1(\ct)\qw}$ and $\gam^2$ 
given by \eqref{gam}
(both depending on the distribution of $\xi$); \citet{Wagner} makes a
numerical calculation of $\mu$ and $\gss$ for his case.

Furthermore, as noted in \cite{Wagner},
$s_1(T)\le s(T)\le |T|s_1(T)$ for any tree (an arbitrary subtree is a
fringe subtree of some subtree containing the root), and thus the asymptotic
normality \eqref{wag} holds for $\log s(\ctn)$ too.

Similarly, the  example by \citet[pp.~78--79]{Wagner} on the average 
size of a  subtree containg the root 
generalizes to arbitrary \cGWt{s} (with $\E\xi^2<\infty$), showing that  
the average size is asymptotically normal with expectation $\sim \mu n$
and variance $\sim \gam^2 $ for some $\mu>0$ and $\gam^2$;
we omit the details. We conjecture that the same is true for the average
size of an arbitrary subtree, as shown in \cite{Wagner} for the case
considered there.
(Note that a uniformly random arbitrary subtree thus is much larger than a
uniformly random fringe subtree, see \refR{Rwiener}.)
\end{example}

\begin{example}
  Another example by \citet{Wagner} is the number of nodes  whose children
  all are leaves (\ie, no grandchildren; \cf{} \refE{Eprotected}).
This is $F(T)$ with $f(T):=\ett{\text{$T$ has no nodes of depth $>1$}}$.
This is a bounded local functional, so \refT{Tlocal} applies and shows
asymptotic normality of $F(\ctn)$ for any \cGWt{} 
with $\E\xi=1$ and $\gss<\infty$, generalizing the result by \citet{Wagner}.
Moreover, $f(T)=1$ only if $T$ is a star, and thus 
$\E f(\ctn)=p_{n-1} p_0^{n-1}/\PP(|\ct|=n)=O\bigpar{n^{3/2}p_0^n}$ so
  \eqref{t1v1}--\eqref{t1v2} hold and \refT{T1} applies too.
\end{example}

\section{Galton--Watson trees}\label{Strees}

Given a random nonnegative integer-valued random variable $\xi$, 
with distribution $\cL(\xi)$,
the \emph{\GWt}
$\cT$
with offspring distribution $\cL(\xi)$ is constructed recursively by
starting with a root and 
giving each node a number of children that is a new copy of $\xi$,
independent of the numbers of children of the other nodes.
(This is thus the family tree of a \GWp, see \eg{} \cite{AthreyaN}.) Obviously, 
only the distribution of $\xi$ matters; we sometimes abuse language and say
that $\cT$ has offspring distribution $\xi$.
We assume that $\PP(\xi=0)>0$ (otherwise the tree is \as{} infinite).

Furthermore, let $\ctn$ be $\cT$ conditioned on having
exactly $n$ nodes; this is called a \emph{\cGWt}.
(We consider only $n$ such that $\PP(|\cT|=n)>0$.)

\begin{remark}
  It is well-known that the \GWt{} $\cT$ is \as{} finite if and only if 
$\E\xi\le1$, see   \cite{AthreyaN}. 
We will in this paper assume
that $\E\xi=1$, the \emph{critical} case. In most cases, but not all,
a \cGWt{} with an offspring distribution $\xi'$ with an expectation
$\E\xi'\neq1$ is equivalent to a \cGWt{} with another offspring distribution
$\xi$ satisfying $\E\xi=1$, so this is only a minor restriction. See \eg{}
\cite{SJ264} for details.
\end{remark}

\begin{remark}\label{RGWdegrees}
  The degree sequence of the \GWt{} $\cT$ (when finite) equals a sequence
  $\xi_1,\xi_2,\dots$ of independent copies of $\xi$, truncated at the
  unique place making the sequence a degree sequence of a tree, \cf{}
  \eqref{ld}. This follows immediately
from (and is equivalent to) the definition of $\cT$.
The degree sequence of the \cGWt{} $\ctn$ 
is more complicated and will be described in \refS{Sprel}.
\end{remark}

\begin{remark}\label{Rfinite}
  For any given $n$, $\st_n$ is finite, so there is only a finite number of
  possible realizations of $\ctn$. Hence, for any functional $f$, the random
  variables $f(\ctn)$ and $F(\ctn)$ are bounded for each $n$; in particular
  they always have finite expectations and higher moments.
\end{remark}

Conditioned \GWt{s} are also known as (a special case of) 
\emph{\sgrt{s}}, see \eg{} \cite{SJ264}.

\section{Preliminaries and more notation}\label{Sprel}

We assume throughout the paper
that $f:\st\to\bbR$ is a given functional on trees, and that 
$F$ is the corresponding subtree sum given by \eqref{F}.
We assume further that $\cT$ [$\cT_n$] is a [conditioned]
\GWt{} with  a given offspring
distribution $\xi$, with $\E\xi=1$ and $0<\gss:=\Var\xi<\infty$.
We let $\xi_1,\xi_2,\dots$ be a sequence of independent copies of $\xi$, and
let
\begin{equation}\label{sn}
  S_n\=\sumin \xi_i.
\end{equation}
We denote the probability distribution of $\xi$ by $(p_k)_0^\infty$, \ie,
$p_k:=\PP(\xi=k)$. 

\subsectionx
We recall the local limit theorem, see \eg{} 
\cite[Theorem 1.4.2]{Kolchin} or 
\cite[Theorem VII.1]{Petrov},
which in our setting can be stated as follows.
Recall that the 
\emph{span} of an integer-valued random variable $\xi$ 
is the largest integer $h$ such that $\xi\in a+h\bbZ$ 
\as{} for some $a\in\bbZ$; we will only consider $\xi$ with $\PP(\xi=0)>0$
and then the span is the largest integer $h$ such that $\xi/h\in\bbZ$ a.s.,
\ie, the greatest common divisor of \set{n:\PP(\xi=n)>0}.
(Typically, $h=1$, but we have for example $h=2$ in the case of full binary
trees, with $p_0=p_2=1/2$.)
In the case we are interested in, the local limit theorem can be stated as
follows. 
\begin{lemma}\label{LLT}
  Suppose that $\xi$ is an integer-valued random variable with 
$\PP(\xi=0)>0$, 
$\E\xi=1$, 
$0<\gss\=\Var\xi<\infty$ and span $h$.
Then, as \ntoo, uniformly in all $m\in h\bbZ$,
\begin{equation}\label{llt}
  \PP(S_n=m)=\frac{h}{\sqrt{2\pi\gss n}} \Bigpar{e^{-(m-n)^2/(2n\gss)}+o(1)}.
\end{equation}
\nopf
\end{lemma}
In particular, which we will use repeatedly, as \ntoo{} with
$n\equiv 1 \pmod h$,
\begin{equation}\label{snn}
  \PP(S_n=n-1)\sim\frac{h}{\sqrt{2\pi\gss }}\,n\qqw.
\end{equation}

\subsectionx
As said above, a tree in $\st$ is uniquely described by its \ds{} $\ddn$.
We may thus define the functional $f$ also on finite nonnegative integer
sequences $\ddn$, $n\ge1$, by
\begin{align}\label{fdd}
f\ddn\=
\begin{cases}
f(T), &\text{$\ddn$ is the \ds{} of a tree $T$},\\
0, & \text{otherwise (\ie,  (\ref{ld}) is not satisfied)}.
\end{cases}
\raisetag{0.9\baselineskip}
\end{align}

If $T$ has \ds{} $\ddn$, and its nodes are numbered $v_1,\dots,v_n$ in
depth-first order so $d_i$ is the degree of $v_i$, then the subtree
$T_{v_i}$ has degree sequence $(d_i,d_{i+1},\dots,d_{i+k-1})$, where $k\le
n-i+1$ is the unique index such that $\ddil$ is a \ds{} of a tree, \ie,
satisfies \eqref{ld}.
By the definition \eqref{fdd}, we thus can write \eqref{F} as
\begin{equation}\label{Fa}
F(T) = \sum_{1\le i\le j\le n}f(d_i,\dots,d_j)
=\sum_{k=1}^n\sum_{i=1}^{n-k+1} f\ddil.
\end{equation}
Moreover, if we regard $\ddn$ as a cyclic sequence and allow wrapping around
by defining $d_{n+i}\=d_i$, we also have the more symmetric formula
\begin{equation}\label{Fc}
F(T) 
=\sum_{k=1}^n\sum_{i=1}^{n} f\ddil.
\end{equation}
The difference from \eqref{Fa} is that we have added some terms
$f(d_i,\dots,d_{i+k-1-n})$ where the indices wrap around, but these terms
all vanish by definition because 
$(d_i,\dots,d_{i+k-1-n})$ is never a \ds.
(The subtree with root $v_i$ is completed at the latest by $v_n$;
this also follows from \eqref{ld}.)

It is a well-known fact, see \eg{} \cite[Corollary 15.4]{SJ264},
that up to a cyclic shift, the \ds{} $\ddn$ of the \cGWt{} $\ctn$ has the
same distribution as $\bigpar{(\xi_1,\dots,\xi_n)\mid\xi_1+\dots+\xi_n=n-1}$.
Since \eqref{Fc} is invariant under cyclic shifts of $\ddn$, it follows
that, recalling \eqref{sn},
\begin{equation}\label{Fd}
  F(\ctn)
\eqd \lrpar{\sum_{k=1}^n\sum_{i=1}^{n} f(\xi_i,\dots,\xi_{i+k-1\bmod n})
\Bigm| S_n=n-1},
\end{equation}
where $j\bmod n$ denotes the index in \setn{} that is congruent to $j$
modulo $n$.

\subsectionx
We let, for $k\ge1$, $f_k$ be $f$ restricted to $\st_k$; more precisely, we 
define $f_k$ for all trees $T\in\st$ by $f_k(T)\=f(T)$ if $|T|=k$ and
$f_k(T)\=0$ otherwise. 
In other words,
\begin{equation}\label{fk}
  f_k(T)\=f(T)\cdot\ett{|T|=k}.
\end{equation}
Extended to integer sequences as in \eqref{fdd}, this
means that
\begin{equation}\label{fkd}
  f_k\ddn=f\ddn\cdot\ett{n=k}.
\end{equation}
Note that $\st_k$ is a finite set; thus $f_k$ is always a bounded function
for each $k$.

We further let, for $k\ge1$ and any tree $T$, with degree sequence $\ddn$,
\begin{equation}\label{Fk}
  F_k(T)\=F(T;f_k)
=\sum_{i=1}^{n-k+1} f_k\ddil.
\end{equation}
(We can also let the sum extend to $n$, wrapping around $d_i$ as in
\eqref{Fc}.) 
Obviously, 
\begin{align}\label{Fsum}
f(T)=\sumki f_k(T)&&&\text{and}&& F(T)=\sumki F_k(T)  
\end{align}
for any tree $T$,
where in both sums it suffices to consider $k\le|T|$ since
the summands vanish for $k>|T|$.

\subsectionx
It is well-known (see \citet{Otter}, or \cite[Theorem 15.5]{SJ264} and the
further references given there)
that for any $k\ge1$,
\begin{equation}\label{ptk}
{\PP(|\cT|=n)}=\frac{1}n\PP(S_n=n-1).  
\end{equation}
Hence, by \eqref{snn}, as \ntoo{} with $n\equiv 1\pmod h$,
see \citet{Kolchin},
\begin{equation}\label{pett}
\PP(|\cT|=n)\sim \frac{h}{\sqrt{2\pi\gss}} n^{-3/2}.
\end{equation}
In particular, $\PP(|\cT|=n)>0$ for all large $n$ with $n\equiv1\pmod h$,

We sometimes use the notation
\begin{equation}
  \label{pik}
\pi_n := \PP(|\cT|=n),
\end{equation}
recalling \eqref{ptk}--\eqref{pett}.

\section{Expectations}\label{Sexp}

We begin the proof of \refT{T1} by calculating the expectation $\E F(\ctn)$,
using \eqref{Fd} which converts this into a problem on expectations of
functionals of a sequence of \iid{} variables conditioned on their sum.
Results of this type have been studied before under various conditions, see 
for example \citet{ZabellPhD,Zabell80,Zabell}, \citet{Swensen} and
\citet{SJ132}. In particular, the results (and methods) of \citet{Zabell} are
closely related and partly overlapping (but the setup there is somewhat
different). 

We assume throughout the paper that $\xi=1$ and $0<\gss=\Var\xi<\infty$.
For simplicity we also assume in some proofs in the sequel that the span
$h$ of the offspring distribution is 1, 
omitting the minor (and standard) modifications in the general case.
All statements are true also for $h>1$.
(Note that when $h>1$, the \GWt{} $\cT$ always
has order $|\cT|\equiv 1\pmod h$, and thus we only consider $k,n\equiv
1\pmod h$. The modifications when $h>1$ consist in using the periodicity
of the characteristic function $\gf(t)$ and integrating only over $|t|<\pi/h$
in, for example, \eqref{ada};
we leave the details to the reader.)
We assume further tacitly that $n$ is so large that $\PP(|\cT|=n)>0$, \cf{}
\eqref{pett}.

By \eqref{Fd} and symmetry, 
\begin{equation}\label{efn}
  \E F(\ctn) = n \sumkn \E\bigpar{f\xik\mid S_n=n-1}.
\end{equation}

We consider first the expectation of each $F_k(\ctn)$ separately,
recalling \eqref{Fsum}. Note that each $f_k$ is bounded, and thus 
trivially $\E|f_k(\cT)|<\infty$. 

\begin{lemma}\label{Lefkn}
  If\/ $1\le k\le n$, then
\begin{equation}\label{lefkn}
  \begin{split}
  \E F_k(\ctn)& 
=n\frac{\PP\xpar{S_{n-k}=n-k}}{\PP(S_n=n-1)} \E{f_k(\cT)}.
  \end{split}
\end{equation}
\end{lemma}

\begin{proof}
If $f_k\xik\neq0$, then $S_k\=\xi_1+\dots+\xi_k=k-1$ by \eqref{ld}.
Consequently, for every $n\ge k$,
by \eqref{efn}, and the fact that the $\xi_i$ are \iid,
\begin{equation}\label{efkn}
  \begin{split}
  \E F_k(\ctn)& = n  \E\bigpar{f_k\xik\mid S_n=n-1}
\\&
=n\frac{\E\bigpar{f_k\xik\cdot\ett{S_n=n-1}}}{\PP(S_n=n-1)}
\\&
=n\frac{\E\bigpar{f_k\xik\cdot\ett{S_n-S_k=n-k}}}{\PP(S_n=n-1)}
\\&
=n\frac{\E{f_k\xik}\cdot\PP\xpar{S_n-S_k=n-k}}{\PP(S_n=n-1)}
\\&
=n\frac{\PP\xpar{S_{n-k}=n-k}}{\PP(S_n=n-1)} \E{f_k\xik}.
  \end{split}
\end{equation}
The result \eqref{lefkn} now follows by \refR{RGWdegrees}, 
which implies that, 
recalling again that
by definition
$f_k(T)=0$ unless $|T|=k$,
\begin{equation}\label{efkkt}
  \E f_k\xik=\E f_k(\cT).
\end{equation}

For future use, we give also an alternative derivation of \eqref{efkkt}. 
By taking $n=k$ in \eqref{efkn}, we obtain 
\begin{equation}\label{efkk}
  \begin{split}
  \E F_k(\cT_k)& 
=\frac{k}{\PP(S_k=k-1)} \E{f_k\xik}.
  \end{split}
\end{equation}
(We may assume  that $\PP(|\cT|=k)>0$, since the result trivially
is true if $\PP(|\cT|=k)=0$, when $f_k(\cT)=0$  a.s.)
Furthermore, since $f_k(T)=0$ unless $|T|=k$,
\eqref{F} yields $F_k(\cT_k)=f_k(\cT_k)$ and
\begin{equation}\label{efkk2}
  \begin{split}
\E F_k(\cT_k)=\E f_k(\cT_k)	
=\E\bigpar{ f_k(\cT)\mid |\cT|=k}
=\frac{\E f_k(\cT)}{\PP(|\cT|=k)}.
  \end{split}
\end{equation}
Finally,  
recalling \eqref{ptk} (which also follows by taking $f_k(T)=1$ in
\eqref{efkk}),
\eqref{efkk}--\eqref{efkk2} yield
\eqref{efkkt}. 
\end{proof}

\begin{lemma}\label{L2}     
  \begin{xenumerate}
\item 
Uniformly for all $k$  with\/ $1\le k\le n/2$,
as \ntoo,
\begin{equation}\label{l2}
  \begin{split}
\frac{\PP\xpar{S_{n-k}=n-k}}{\PP(S_n=n-1)} 
= 1+ O\parfrac{k}{n}+o\bigpar{n\qqw}.
  \end{split}
\end{equation}
\item 
If $n/2 < k\le n$, then 
\begin{equation}\label{l2b}
  \begin{split}
\frac{\PP\xpar{S_{n-k}=n-k}}{\PP(S_n=n-1)} 
= O\parfrac{n\qq}{(n-k+1)\qq}.
  \end{split}
\end{equation}
  \end{xenumerate}
\end{lemma}

If $\xi$ has a finite third moment, this follows easily from
the refined local limit theorem in \cite[Theorem VII.13]{Petrov}.
Since we do not assume this, we have to work harder and take advantage of
some cancellation.

\begin{proof}
\pfitem{i}
We let $\gf(t)\=\E e^{\ii t\xi}$ be the \chf{} of $\xi$, and
  $\tgf(t)\=e^{-\ii t}\gf(t)$ the \chf{} of the centred variable
  $\txi:=\xi-\E\xi=\xi-1$. 

We begin with a standard estimate.
Since $\E\txi=0$ and $\Var(\txi)=\gss<\infty$, we have
\begin{equation}\label{tgf}
  \tgf(t)=1-\tfrac12\gss t^2+o(t^2)
\qquad\text{as $|t|\to0$}.
\end{equation}
It follows that $|\gf(t)|=|\tgf(t)|<e^{-\gss t^2/3}$ for 
$|t|\le \cc \ccdef\cca$.
Furthermore,  assuming that $\xi$ has span $h=1$, $|\gf(t)|<1$ for
$0<|t|\le\pi$, so by compactness, $|\gf(t)|\le 1-\cc$ for 
$\cca\le|t|\le\pi$. It follows that
\begin{equation}\label{ccgf}
  |\gf(t)|=|\tgf(t)|\le e^{-\cc t^2},
\qquad |t|\le\pi.
\ccdef\ccgf
\end{equation}

To estimate the ratio in \eqref{l2}, we note first that 
by \eqref{snn}, it suffices to estimate the difference
$\PP(S_{n-k}=n-k)-\PP(S_{n}=n-1)$. We do this in two steps.

First, consider the difference $\PP(S_{n-k}=n-k)-\PP(S_{n-1}=n-1)$.
By Fourier inversion,
	\begin{multline}
\label{ada}
\PP(S_{n-k}=n-k)-\PP(S_{n-1}=n-1)
=
\intpipi\Bigpar{\tgf^{n-k}(t)-\tgf^{n-1}(t)}\dd t
\\
=\intpipi\Bigpar{1-\tgf^{k-1}(t)}\tgf^{n-k}(t)\dd t	
.	\end{multline}
By \eqref{tgf}, $\tgf(t)=1+O(t^2)$, and thus, using also $|\tgf(t)|\le1$,
for all $j\ge0$ and $t\in\bbR$,
\begin{equation}
  \label{midsommar}
|\tgf^{j}(t)-1|=O(jt^2).
\end{equation}
Furthermore, by \eqref{ccgf} and $k\le n/2$, for $|t|\le\pi$,
\begin{equation}
  \label{johd}
|\tgf^{n-k}(t)|\le \exp\xpar{-\ccgf(n-k)t^2} \le \exp\xpar{-\cc nt^2}.
\ccdef\ccjohd
\end{equation}
Consequently, \eqref{ada} yields
\begin{equation}\label{jw}
  \begin{split}
\bigabs{\PP(S_{n-k}=n-k)&-\PP(S_{n-1}=n-1)}
\\
&\ll
\intpipi  kt^2 \xexp(-\ccx n t^2)\dd t
\ll k n^{-3/2}  
.	
  \end{split}
\end{equation}

Next, 
consider  $\PP(S_{n-1}=n-1)-\PP(S_{n}=n-1)$.
By Fourier inversion,
\begin{equation}\label{sw}
  \begin{split}
&\PP(S_{n-1}=n-1)-\PP(S_{n}=n-1)
=
\intpipi e^{-\ii(n-1)t}\Bigpar{\gf^{n-1}(t)-\gf^{n}(t)}\dd t
\\&\qquad
=\intpipi\Bigpar{1-\gf(t)}\tgf^{n-1}(t)\dd t	
\\&\qquad
=\frac{-\ii}{2\pi}\intpipix{t}\tgf^{n-1}(t)\dd t	
+\intpipi\Bigpar{1+\ii t-\gf(t)}\tgf^{n-1}(t)\dd t	
.	
  \end{split}
\raisetag{1.2\baselineskip}
\end{equation}
Since $\gf(t)=1+\ii t+O(t^2)$, the second integral in \eqref{sw} is
$O(n^{-3/2})$ by the argument in \eqref{jw}.
For the first integral, we make the change of variable $t=x/\sqrt n$:
\begin{equation}\label{maa}
\intpipix t \tgf^{n-1}(t)\dd t
=
\frac{1}n\int_{-\pi\sqrt n}^{\pi\sqrt n} x \tgf^{n-1}\parfrac{x}{\sqrt n}\dd x
.
\end{equation}
We have $\tgf^{n-1}(x/\sqrt n)\to e^{-\gss x^2/2}$ as \ntoo{} for every $x$
by \eqref{tgf},
and thus by dominated convergence (justified by \eqref{ccgf}),
\begin{equation}\label{mab}
  \begin{split}
\int_{-\pi\sqrt n}^{\pi\sqrt n} x \tgf^{n-1}\parfrac{x}{\sqrt n}\dd x
\to
\intoooo x e^{-\gss x^2/2}\dd x =0.
  \end{split}
\end{equation}
Consequently, the expressions in \eqref{maa} are $o(1/n)$,
and
\eqref{sw} yields
\begin{equation}
\PP(S_{n-1}=n-1)-\PP(S_{n}=n-1)=o\bigpar{n^{-1}}.
\end{equation}
This and \eqref{jw} yield, together with \eqref{snn},
\begin{equation}
  \frac{|\PP(S_{n-k}=n-k)-\PP(S_{n}=n-1)|}{\PP(S_n=n-1)}
\ll \frac{ k n^{-3/2}+o(n\qw)}{n\qqw},
\end{equation}
and the result follows.

\pfitem{ii}
We use \eqref{snn} together with the similar estimate,
also from \refL{LLT},
$\PP(S_{n-k}=n-k)=O\bigpar{(n-k+1)\qqw}$.
\end{proof}

\begin{proof}[Proof of \refT{T1}\ref{T1e}]
Let $a_k\=|\E f(\ctk)|=|\E f_k(\ctk)|$; thus by assumption $a_k\to0$ as \ktoo.
Moreover, by \eqref{efkk2} and \eqref{pett}, 
\begin{equation}\label{chua}
  |\E f_k(\cT)|={|\E f_k(\ctk)|\PP(|\cT|=k)}
=a_k \PP(|\cT|=k)
=O\bigpar{a_k k^{-3/2}}.
\end{equation}

By \eqref{Fsum} and \refL{Lefkn},
\begin{equation}\label{chu}
  \begin{split}
&  \frac1n\E F(\ctn) -\E f(\cT)
=
\sumki \Bigpar{\frac{1}n\E F_k(\ctn)-\E f_k(\cT)}	
\\& \qquad
=
\sumkn\lrpar{\frac{\PP\xpar{S_{n-k}=n-k}}{\PP(S_n=n-1)}-1} \E{f_k(\cT)}
-\sum_{k=n+1}^\infty \E{f_k(\cT)}.
  \end{split}
\end{equation}
We split the expression in \eqref{chu}
 into three parts. First, for $k\le n/2$ we use \refL{L2}(i) and
obtain, using \eqref{chua} and $a_k\to0$ as \ktoo,
\begin{equation}\label{chux}
  \begin{split}
\sum_{k\le n/2}\lrabs{\frac{\PP\xpar{S_{n-k}=n-k}}{\PP(S_n=n-1)}-1}
\bigabs{\E{f_k(\cT)}}
\ll
\sum_{k\le n/2}\lrpar{\frac{k}n+o\bigpar{n\qqw}} a_k k\qqcw
\\
\ll n\qw \sum_{k\le n}  a_k k\qqw + o\bigpar{n\qqw}
=o\bigpar{n\qqw}.
  \end{split}
\end{equation}
For $n/2<k\le n$, we use \refL{L2}(ii),
yielding
\begin{equation}\label{chuy}
  \begin{split}
\sum_{n/2<k\le n}\lrabs{\frac{\PP\xpar{S_{n-k}=n-k}}{\PP(S_n=n-1)}-1}
\bigabs{\E{f_k(\cT)}}
\ll
\sum_{n/2<k\le n}\frac{n\qq}{(n-k+1)\qq} a_k k\qqcw
\\
\ll 
 n\qw \max _{k\ge n/2} a_k\sum_{n/2<k\le n}\frac{1}{(n-k+1)\qq} 
=o\bigpar{n\qqw}.
  \end{split}
\end{equation}
Finally, for $k>n$ we have by \eqref{chua}
\begin{equation}\label{chuz}
  \begin{split}
\sum_{k>n}\bigabs{\E{f_k(\cT)}}
\ll
 \max _{k> n} a_k\sum_{k> n} k\qqcw
=o(1)\cdot 
\sum_{k> n} k\qqcw
=o\bigpar{n\qqw}.
  \end{split}
\end{equation}
The result follows by \eqref{chu}--\eqref{chuz}.
\end{proof}

\begin{remark} 
  \label{Rp}
Trivial modifications in the proof above show that if $\E|f(\cT)|<\infty$
and $|\E f(\ctk)|=o(k\qq)$, then  
$\E F(\ctn) =n\mu+o(n)$, so \eqref{t2af} holds.
Moreover, the quenched version \eqref{t2qf} holds too; this follows easily
from \eqref{t2af} together with \eqref{t2qf} applied to truncations of $f$.
(We omit the details.)
\end{remark}

If we assume further moment conditions on $\xi$, we can improve 
the error term in \eqref{l2} and thus in \eqref{t1e}.
(Cf.\ \citet[Theorem 4]{Zabell}.)

\begin{lemma}\label{L23}
If\/ $\E\xi^{2+\gd}<\infty$ with $0<\gd\le1$, then,
uniformly for all $k$ and $n$ with\/ $1\le k\le n/2$,
\begin{equation}\label{l23}
  \begin{split}
\frac{\PP\xpar{S_{n-k}=n-k}}{\PP(S_n=n-1)} 
= 1+ O\parfrac{k}{n}
+O\bigpar{n^{-(1+\gd)/2}}.
  \end{split}
\end{equation}
\end{lemma}

\begin{proof}
This follows by minor modifications in the proof of \refL{L2}(i).
We now have
\begin{equation}\label{tgf3}
  \tgf(t)=1-\tfrac12\gss t^2+O(|t|^{2+\gd})
\end{equation}
which leads to
\begin{equation}
x \tgf^{n-1}\parfrac{x}{\sqrt n}
= x e^{-\gss x^2/2}
\lrpar{1+O\parfrac{|x|^{2+\gd}}{n^{\gd/2}} +O\parfrac{x^2}n},  
\end{equation}
for $|x|\le n^{\gd/6}$, at least.
It follows 
(using \eqref{johd} for $x>n^{\gd/6}$)
that the first integral in \eqref{mab} is $O\bigpar{n^{-\gd/2}}$.
The rest is as before.
\end{proof}

\begin{theorem}
  \label{T13}
Suppose, in addition to the assumptions of \refT{T1}\ref{T1e}, that
$0<\gd<1$ and that
$\E\xi^{2+\gd}<\infty$ 
and $\E f(\ctn)=O(n^{-\gd/2})$. 
Then
\begin{equation}\label{t1e3gd}
\E F(\ctn) 
=n\mu+O\bigpar{n^{(1-\gd)/2}}.  
\end{equation}
Similarly,
if $\E\xi^3<\infty$,
$\E f(\ctn)=O(n\qqw)$ and $\sumni |\E f(\ctn)|n\qqw<\infty$,
then
\begin{equation}\label{t1e3}
\E F(\ctn) 
=n\mu+O(1).  
\end{equation}
\end{theorem}

\begin{proof}
  As the proof of \refT{T1}\ref{T1e} above, using \eqref{l23} and the
  assumptions on $\E f(\ctn)$. We omit the details.
\end{proof}

\begin{remark}\label{R23}
  In fact, if $\E\xi^3<\infty$,
by including the next terms explicitly in the calculations in the
  proof of \refL{L2}(i), it is easily shown that for every fixed $k$, as \ntoo,
\begin{equation}\label{r23}
  \begin{split}
\frac{\PP\xpar{S_{n-k}=n-k}}{\PP(S_n=n-1)} 
= 1+ \frac{1}{2n}\bigpar{k+\gs\qww-\gk_3\gs^{-4}}+o\bigpar{n\qw},
  \end{split}
\end{equation}
where $\gk_3=\E(\xi-\E\xi)^3$ is the third cumulant of $\xi$.
(If $\E\xi^4<\infty$, this also follows easily from 
\cite[Theorem VII.13]{Petrov}.)
Hence, if for simplicity $f$ has finite support,
\eqref{chu} yields, with $\mu:=\E f(\cT)$ as above,
\begin{equation}\label{r23b}
  \begin{split}
\E F(\ctn) 
=n\mu
+
 \tfrac{1}{2}\E\bigpar{|\cT|f(\cT)}
+\tfrac12\bigpar{\gs\qww-\gk_3\gs^{-4}}\mu
+o(1).
  \end{split}
\end{equation}
We leave it to the reader to find more general conditions on $f$ for
\eqref{r23b} to hold.
\end{remark}

The following  example (adapted from \cite{Zabell}) shows that 
the sharper results in \refL{L23} and \refT{T13} do not hold without
the extra moment assumption on $\xi$.

\begin{example}
This is a discrete version of \cite[Examples 5--6]{Zabell}.
  Consider, as in \refE{Eleaves}, $f(T)=\ett{|T|=1}$.
Suppose that $p_k=\PP(\xi=k)=ak^{-\ga}$ for $k\ge2$ for some $a>0$ and
$\ga\in(3,4)$. (With $p_0,p_1$ adjusted so that $\sum_k p_k=1$ and
$\E\xi=1$; this is obviously
possible if $a$ is small.)
Then $\E\xi^r<\infty\iff r<\ga-1$; in particular, $\E\xi^2<\infty$ but
$\E\xi^3=\infty$.
It can be verified that
\begin{equation}
  \gf(t) =1+ \ii t -\tfrac12\E\xi^2 t^2 + a\gG(1-\ga)(-\ii t)^{\ga-1} + O(t^3),
\end{equation}
see \eg{} \cite[25.12.12]{NIST} or \cite[Theorem VI.7]{FlajoletS}.
It follows that, for $|t|\le \cc$,
\begin{equation}
\log \tgf(t) =  -\tfrac12\gss t^2 + a\gG(1-\ga)(-\ii t)^{\ga-1} + O(t^3),
\end{equation}
and hence, for $|x|\le n^{(\ga-3)/6}$, by a simple calculation,
\begin{multline*}
 \tgf^{n-1}\parfrac{x}{\sqrt n} =  
e^{-\gss x^2/2}\biggl(1 + a\gG(1-\ga)(-\ii x)^{\ga-1}n^{(3-\ga)/2} 
\\
+ O\Bigparfrac{x^2+|x|^3}{n\qq}+O\Bigparfrac{|x|^{2\ga-2}}{n^{\ga-3}}\biggr).
\end{multline*}
Using this in \eqref{maa}, it is easy to obtain
\begin{equation*}
  \begin{split}
-\ii&\intpipix t\tgf^{n-1}(t)\dd t
\\&
=\intoooo e^{-\gss x^2/2}\lrpar{\frac{-\ii x}n 
+ a\gG(1-\ga)(-\ii x)^{\ga}n^{(1-\ga)/2}}\dd x + O\bigpar{n\qqcw+n^{2-\ga}}
\\ &
= 2  a\gG(1-\ga)\cos\frac{\pi\ga}2 n^{(1-\ga)/2}\intoo x^\ga e^{-\gss x^2/2}
\dd x + O\bigpar{n\qqcw+n^{2-\ga}}
\\ &
= b n^{(1-\ga)/2} + o\bigpar{n^{(1-\ga)/2}}
  \end{split}
\end{equation*}
for some $b\neq0$.
Using this, instead of \eqref{maa}--\eqref{mab}, in the 
proof of \refL{L2}, leads to the estimate,
for some $c\neq0$,
\begin{equation}
  \begin{split}
\frac{\PP\xpar{S_{n-1}=n-1}}{\PP(S_n=n-1)} 
= 1+ c n^{1-\ga/2}+o\bigpar{n^{1-\ga/2}},
  \end{split}
\end{equation}
and thus by \refL{Lefkn}
\begin{equation}
  \begin{split}
\E F(\ctn) =np_0+ cp_0 n^{2-\ga/2}+o\bigpar{n^{2-\ga/2}},
  \end{split}
\end{equation}
showing that without further assumptions, the error term $o(n\qqw)$ in
 \refT{T1}\ref{T1e} is essentially best possible.

We can also take $\ga=4$ in this example; then 
\begin{equation}
  \gf(t) =1+ \ii t -\tfrac12\E\xi^2 t^2 + \tfrac{\ii}6 a t^{3}\log|t| + O(t^3),
\end{equation}
and similar calculations lead to, for some $c\neq0$,
\begin{equation}
  \begin{split}
\E F(\ctn) =np_0+ cp_0 \log n+O(1).
  \end{split}
\end{equation}
\end{example}

We end this section 
with a result on the expectation $\E f(\ctn)$ in the case of a local
functional $f$. We first state an estimate similar to \refL{L2} (but
somewhat coarser and simpler); it can be refined but the present version is
enough for our needs.
(If $\xi$ has a finite third moment, it too, and more, follows easily from
the refined local limit theorem in \cite[Theorem VII.13]{Petrov}.)

\begin{lemma}
  \label{Lzw}
For any integers $w,z\ge0$ with $z\le n/2$,
\begin{equation}
  \frac{\PP(S_{n-z}=n-z-w)}{\PP(S_{n}=n-1)}
= 1 + O\Bigpar{\frac{w+z+1}{n\qq}}.
\end{equation}
\end{lemma}
\begin{proof}
  We argue as in the proof of \refL{L2} and obtain,
recalling \eqref{midsommar} and \eqref{johd},
\begin{equation*}
  \begin{split}
&\PP(S_{n-z}=n-z-w)-\PP(S_{n}=n-1)
=
\intpipi\Bigpar{\tgf^{n-z}(t)e^{\ii wt}-\tgf^{n}(t)e^{\ii t}}\dd t
\\&\qquad
=
\intpipi\tgf^{n-z}(t)\Bigpar{e^{\ii wt}-\tgf^{z}(t)e^{\ii t}}\dd t
\\&\qquad
=
\intpipi\tgf^{n-z}(t)
\Bigpar{1+O\bigpar{w|t|}-\bigpar{1+O\bigpar{zt^2}+O\bigpar{|t|}}}\dd t
\\&\qquad
= \intpipix e^{-\ccjohd nt^2}
O\bigpar{w|t|+z|t|+|t|}\dd t
=O\Bigparfrac{w+z+1}n.
  \end{split}
\end{equation*}
The result follows by division by $\PP(S_n=n-1)$, using \eqref{snn}.
\end{proof}

Let $\hct$ be the \emph{size-biased} \GWt{} defined by \citet{Kesten},
see also \citet{AldousII}, \citet{AldousPitman},  \citet{LPP} and
\citet{SJ264}; this is a random infinite tree, whose 
distribution can be described
in terms of the truncations $\hct\MM$ by
\begin{equation}\label{cant}
  \PP(\hct\MM=T) =w_M(T)\PP(\cT\MM=T),\qquad T\in\st,
\end{equation}
where $w_M(T)$ denotes the number of nodes of depth 
$M$ in $T$. Then $\ctn\dto\hct$ as \ntoo{} in the appropriate (local)
topology, as shown by \citet{Kennedy} and \citet{AldousPitman}, see also 
\citet{SJ264} for details and generalizations. This means that
$\ctn\MM\dto\hct\MM$ for every fixed $M$.
If $f$ is a local functional with cut-off $M$, so $f(T)=f(T\MM)$
for every finite tree $T$, then we define $f$ also for the infinite tree
$\hct$ by $f(\hct):=f(\hct\MM)$. 
It follows that $f(\ctn)=f(\ctn\MM)\dto f(\hct\MM)=f(\hct)$, and thus,
if $f$ furthermore is bounded, that $\E f(\ctn)\to\E f(\hct)$.
We establish an upper bound on the rate of this convergence. (Note that we
do not impose any further moment condition on $\xi$ beyond finite variance.)

\begin{lemma}\label{Lfloc}
  If $f(T)$ is a bounded local functional on $\st$, then
  \begin{equation}
	\E f(\ctn) = \E f(\hct) + O\bigpar{n\qqw}.
  \end{equation}
\end{lemma}
\begin{proof}
Let as above $M\ge1$ be the cut-off of $f$. Let
$T$ be a tree with height $\le M$ and condition on the event that 
$\cT\MM=T$. Then the rest of the tree, more precisely $\cT\setminus\cT\MMi$,
is a random forest consisting of $w:=w_M(T)$ independent copies of $\cT$;
denote this random forest by $\cF_w$.
By an extension of \eqref{ptk} due to \citet{Dwass}, see also 
\citet{Kemperman1,Kemperman2} and \citet{Pitman:enum}, 
\begin{equation}\label{cant1}
  \PP(|\cF_w|=n) =\frac{w}n\PP(S_n=n-w).
\end{equation}
Let $z_k=z_k(T):=\sum_{j=0}^k w_j(T)$, the number of nodes in the first $k$
generations of $T$.
Let further
$\ptmt:=\PP\bigpar{\cT\MM=T}$
and, using \eqref{cant}, $\hptmt:=\PP\bigpar{\hct\MM=T}=w\ptmt$.
It follows that, for $n\ge z_M$,
  \begin{equation*}
	\begin{split}
\PP\bigpar{\cT\MM=T\text{ and }	|\cT|=n}  
&=
\PP\bigpar{\cT\MM=T}\PP\bigpar{|\cF_w|=n-z_{M-1}}
\\&=
\ptmt\frac{w}{n-z_{M-1}}\PP\bigpar{S_{n-z_{M-1}}=n-z_{M-1}-w}
\\&=
\hptmt\frac{\PP\bigpar{S_{n-z_{M-1}}=n-z_{M-1}-w}}{n-z_{M-1}}.
	\end{split}
  \end{equation*}
Hence, recalling \eqref{ptk},
  \begin{equation*}
	\begin{split}
\PP(\ctn\MM=T)
&=\PP\bigpar{\cT\MM=T \mid |\cT|=n}  
\\&
=\frac{\PP\bigpar{\cT\MM=T\text{ and }|\cT|=n}}
{\PP(|\cT|=n)}
\\&
=\hptmt \frac{n}{n-z_{M-1}}
\frac{\PP\bigpar{S_{n-z_{M-1}}=n-z_{M-1}-w}}{\PP\bigpar{S_n=n-1}}
	\end{split}
  \end{equation*}
If $z_{M-1}\le n/2$, we thus obtain by \refL{Lzw},
  \begin{equation}\label{cant2}
	\begin{split}
\PP(\ctn\MM=T)
= \hptmt \Bigpar{1+O\Bigparfrac{z_{M-1}+w}{n\qq}}
= \hptmt \Bigpar{1+O\Bigparfrac{z_{M}}{n\qq}}.
	\end{split}
  \end{equation}
(Incidentally, this proves $\PP(\ctn\MM=T)\to \hptmt=\PP(\hct\MM=T)$, \ie,
  $\ctn\MM\dto\hct\MM$ as asserted above.)

Consequently, since $f$ is bounded, using \eqref{cant2} when $|T|\le n/2$,
\begin{equation}\label{cant3}
  \begin{split}
&\bigabs{\E f(\ctn)-\E f(\hct)}
=
\bigabs{\E f(\ctn\MM)-\E(\hct\MM)}
\\&\quad
=\Bigabs{\sum_{T} f(T)\PP(\ctn\MM=T)-\sum_Tf(T)\PP(\hct\MM=T)}
\\&\quad
\ll
\sum_T \bigabs{\PP(\ctn\MM=T)-\PP(\hct\MM=T)}
\\&\quad
\ll \sum_T \PP(\hct\MM=T) \frac{z_M(T)}{n\qq} 
+ \sum_{|T|>n/2}\bigpar{\PP(\ctn\MM=T)+\PP(\hct\MM=T)}
\\&\quad
= n\qqw \E|\hct\MM| + \PP\bigpar{|\ctn\MM|>n/2}
+ \PP\bigpar{|\hct\MM|>n/2}
\\&\quad
\ll n\qqw \E|\hct\MM| + n\qw \E|\ctn\MM|,
  \end{split}
\raisetag{\baselineskip}
\end{equation}
using Markov's inequality at the final step.
Finally, we observe that
\begin{equation}
\E|\ctn\MM|=\sum_{k=0}^M \E w_k(\ctn) =O(1)
\end{equation}
since $\E w_j(\ctk) = O(j)$ for each $j$, 
see \cite{MM} (assuming that $\xi$ has an exponential moment)
and \cite[Theorem 1.13]{SJ167} (general $\xi$ with $\E\xi^2<\infty$; in fact
it is shown that the estimate holds uniformly in $j$);
see also \cite{SJ250} for further results.
Similarly, or as a consequence,
\begin{equation}
\E|\hct\MM|=\sum_{k=0}^M \E w_k(\hct) =\sum_{k=0}^M \E w_k(\ct)^2  
<\infty.
\end{equation}
Hence \eqref{cant3} yields the estimate $O(n\qqw)$.
\end{proof}

\section{Variances and covariances}\label{Svar}

We next consider the variance of $F(\ctn)$. 
As in \refS{Sexp}, we consider first
the different $F_k(\ctn)$ separately; thus we study variances and
covariances of these sums.
We begin with an exact formula, corresponding to \refL{Lefkn}.

\begin{lemma}\label{Lcov}
If\/ $m\le k$ and $n\ge k+m-1$, then  
{\multlinegap=0pt
\begin{multline}\label{lcov}
\Cov( F_k(\ctn), F_m(\ctn))
=n\frac{\PP\xpar{S_{n-k}=n-k}}{\PP(S_n=n-1)} \E\bigpar{f_k(\cT) F_m(\cT)}
\\\shoveleft{\quad
-n(k+m-1)\frac{\PP\xpar{S_{n-k}=n-k}}{\PP(S_n=n-1)}\cdot
\frac{\PP\xpar{S_{n-m}=n-m}}{\PP(S_n=n-1)}
\E{f_k(\cT)}\E f_m(\cT)}
\\\shoveleft{\quad
+n(n-k-m+1)
\E{f_k(\cT)}\E{f_m(\cT)}
\cdot}
\\\lrpar{\frac{\PP\xpar{S_{n-k-m}=n-k-m+1}}{\PP(S_n=n-1)} 
-\frac{\PP\xpar{S_{n-k}=n-k}}{\PP(S_n=n-1)} 
\frac{\PP\xpar{S_{n-m}=n-m}}{\PP(S_n=n-1)}}.
\end{multline}}%

If\/ $m\le k\le n \le k+m$, we have instead
\begin{multline}\label{lcov2}
\Cov( F_k(\ctn), F_m(\ctn))
=n\frac{\PP\xpar{S_{n-k}=n-k}}{\PP(S_n=n-1)} \E\bigpar{f_k(\cT) F_m(\cT)}
\\
{}-n^2\frac{\PP\xpar{S_{n-k}=n-k}}{\PP(S_n=n-1)}\cdot
\frac{\PP\xpar{S_{n-m}=n-m}}{\PP(S_n=n-1)}
\E{f_k(\cT)}\E f_m(\cT).
\end{multline}
\end{lemma}

Note that by \eqref{efkk2} and \eqref{ptk},
\begin{equation} \label{ek}
\E f_k(\cT)=\PP(|\cT|=k)\E f_k(\ctk) =\frac{\PP(S_{k}=k-1)}{k}\E f_k(\ctk)  
\end{equation}
and similarly (again because $f_k(\cT)=0$ unless $|\cT|=k$)
\begin{equation} \label{ekm}
\E \bigpar{f_k(\cT)F_m(\cT)}
=\PP(|\cT|=k)\E \bigpar{f_k(\ctk)F_m(\ctk)}.
\end{equation}

\begin{proof}
Note first that for $n=k+m-1$ and $n=k+m$, the formulas \eqref{lcov} and
\eqref{lcov2} agree, in the latter case because
$\PP(S_{n-k-m}=n-k-m+1)=P(S_0=1)=0$. Hence it suffices to prove
\eqref{lcov} for $n\ge k+m$ and \eqref{lcov2} for $k \le n\le k+m-1$.

By \eqref{Fd} and symmetry, 
\begin{multline}\label{vfn}
  \E\bigpar{F_k(\ctn)F_m(\ctn)} 
\\= n\sum_{j=0}^{n-1} 
\E\bigpar{f_k\xik f_m(\xi_{j+1},\dots,\xi_{j+m\bmod n})\mid S_n=n-1}
\end{multline}
Consider first the terms with $0\le j\le k-m$; these are the terms with
$\set{j+1,\dots,j+m}\subseteq\set{1,\dots,k}$,
and we see from \eqref{Fk} that if $(\xi_1,\dots,\xi_k)$ is the degree
sequence of a tree, then
\begin{equation}
\sum_{j=0}^{k-m} f_m(\xi_{j+1},\dots,\xi_{j+m})=F_m\xik.
\end{equation}
Hence, if we define $g(T):=f_k(T)F_m(T)$, $T\in\st$, and 
use \eqref{efn} and \refL{Lefkn} (or its proof),
noting that $g_k=g$,
\begin{equation}
  \begin{split}
&\sum_{j=0}^{k-m} 
\E\bigpar{f_k\xik f_m(\xi_{j+1},\dots,\xi_{j+m})\mid S_n=n-1}
\\&\qquad
=\E\bigpar{f_k\xik F_m\xik\mid S_n=n-1}
\\&\qquad
=\E\bigpar{g_k\xik\mid S_n=n-1}
=\frac1n\E G_k(\cT_n)
\\&\qquad
=\frac{\PP(S_{n-k}=n-k)}{\PP(S_{n}=n-1)}\,\E g_k(\cT).
  \end{split}
\end{equation}
This yields the first term on the \rhs{} of \eqref{lcov} and \eqref{lcov2}.

Next, two subtrees of a tree are either disjoint or one is a subtree of the
other. Hence, if $k-m<j<k$ so the index sets 
$\set{1,\dots,k}$ and $\set{j+1,\dots,j+m}$ overlap partly,
$(\xikx)$ and $(\xijmx)$ cannot both be degree sequences of trees (this can
also be seen algebraically from \eqref{ld}), and thus
\begin{equation*}
  f_k\xik f_m\xijm = 0.
\end{equation*}
Hence these terms in the sum in \eqref{vfn} vanish.
The same holds if $n-m < j < n$, with indices taken modulo $n$, when the
index sets again overlap partly (on the other side).

Finally, if $k\le j\le n-m$ (and thus $n\ge k+m$), 
the index sets $\set{1,\dots,k}$ and $\set{j+1,\dots,j+m}$ are disjoint.
By symmetry, the expectation in \eqref{vfn}
is the same for all $j$ in this range, so
we may assume $j=k$, noting that
this term appears $n-k-m+1$ times.
Arguing as in \eqref{efkn}, and using \eqref{efkkt},
\begin{equation}\label{vfnx}
  \begin{split}
&\E\bigpar{f_k\xik f_m(\xi_{k+1},\dots,\xi_{k+m})\mid S_n=n-1}
\\&\quad
=\frac{\E\bigpar{f_k\xik f_m(\xi_{k+1},\dots,\xi_{k+m})
\ett{S_n-S_{k+m}=n-k-m+1}}}
{\PP(S_n=n-1)}
\\&\quad
=\frac{\E{f_k\xik}\E{f_m(\xi_{k+1},\dots,\xi_{k+m})}
\PP\xpar{S_n-S_{k+m}=n-k-m+1}}{\PP(S_n=n-1)}
\\&\quad
=\frac{\PP\xpar{S_{n-k-m}=n-k-m+1}}{\PP(S_n=n-1)}{\E{f_k(\cT)}}\E f_m(\cT).
  \end{split}
\raisetag\baselineskip
\end{equation}

The results 
\eqref{lcov} and \eqref{lcov2}
now follow from \eqref{vfn}--\eqref{vfnx}, 
subtracting the product $\E F_k(\ctn) \E F_m(\ctn)$
which is given by two applications of \eqref{lefkn}.
(Note that in \eqref{lcov}, this is split into two terms.)
\end{proof}

We next estimate one of the factors in \eqref{lcov}, where there typically is
a lot of cancellation. 

\begin{lemma}\label{Lv2}
  \begin{xenumerate}
  \item 
As \ntoo, 
uniformly for all $k\ge0$ and $m\ge0$ with\/ $k+m\le n/2$,
\begin{multline}\label{lv2}
\frac{\PP\xpar{S_{n-k-m}=n-k-m+1}}{\PP(S_n=n-1)} 
-{\frac{\PP\xpar{S_{n-k}=n-k}}{\PP(S_n=n-1)}}
{\frac{\PP\xpar{S_{n-m}=n-m}}{\PP(S_n=n-1)}}
\\
= -\frac{1}{n\gss}+o\Bigparfrac1n+O\Bigpar{\frac{k+m}{n^{3/2}}}
+O\Bigpar{\frac{km}{n^2}}.
\end{multline}

\item 
For all $n\ge1$, $k\ge0$ and $m\ge0$ with\/ $n/2 \le k+m \le n$,
\begin{multline}\label{lv2b}
\frac{\PP\xpar{S_{n-k-m}=n-k-m+1}}{\PP(S_n=n-1)} 
-{\frac{\PP\xpar{S_{n-k}=n-k}}{\PP(S_n=n-1)}}
{\frac{\PP\xpar{S_{n-m}=n-m}}{\PP(S_n=n-1)}}
\\
= O\biggpar{\frac{\min(k,m)\, n\qq}{(n-k-m+1)^{3/2}}}
+O\biggpar{\frac{n\qq}{n-k-m+1}}.
\end{multline}
	  \end{xenumerate}
\end{lemma}

\begin{proof}
\pfitem{i}
By multiplying \eqref{lv2} by $\PP(S_n=n-1)^2$ and using \eqref{snn}, we see
that \eqref{lv2} is equivalent to (assuming $h=1$ for simplicity)
  \begin{multline}\label{lv2x}
\PP\xpar{S_{n-k-m}=n-k-m+1}{\PP(S_n=n-1)} 
\\\shoveright{
-\PP\xpar{S_{n-k}=n-k}\PP\xpar{S_{n-m}=n-m}}
\\
= -\frac{1}{2\pi n^2\gs^4}+o\Bigparfrac1{n^2}+O\Bigpar{\frac{k+m}{n^{5/2}}}
+O\Bigpar{\frac{km}{n^3}}.
  \end{multline}
To prove this, we first obtain by Fourier inversion, recalling
  $\tgf(t)\=e^{-\ii t}\gf(t)$,
  \begin{multline}\label{sol}
\PP\xpar{S_{n-k-m}=n-k-m+1}{\PP(S_n=n-1)} 
\\\shoveright{
-\PP\xpar{S_{n-k}=n-k}\PP\xpar{S_{n-m}=n-m}}
\\\shoveleft{\qquad
= \intpipi \tgf^{n-k-m}(t)e^{-\ii t}\dd t\cdot
\intpipi \tgf^{n}(u)e^{\ii u}\dd u
}
\\\shoveright{
-\intpipi \tgf^{n-k}(t)\dd t \cdot\intpipi \tgf^{n-m}(u)\dd u}
\\
\shoveleft{\qquad
=
\frac{1}{8\pi^2}\intpipix\intpipix \tgf^{n-k-m}(t)e^{-\ii t}
 \tgf^{n-k-m}(u)e^{-\ii u} 
\times}
\\
\Bigpar{\tgf^k(t)e^{\ii t}-\tgf^k(u)e^{\ii u}}
\Bigpar{\tgf^m(t)e^{\ii t}-\tgf^m(u)e^{\ii u}}
\dd t\dd u.
  \end{multline}
For all $k\ge 0$, we have by \eqref{midsommar}
$|\tgf^k(t)e^{\ii t}-1|=O(|t|+kt^2)$
and thus
\begin{equation}\label{sola}
  |\tgf^k(t)e^{\ii t}-\tgf^k(u)e^{\ii u}|
=O(|t|+kt^2+|u|+ku^2);
\end{equation}
similarly,
\begin{equation}\label{solb}
\Bigabs{\bigpar{\tgf^m(t)e^{\ii t}-\tgf^m(u)e^{\ii u}}
- \tgf^m(t)\tgf^m(u)\bigpar{e^{\ii t}-e^{\ii u}}
}
=O(mt^2+mu^2).
\end{equation}
Furthermore, 
if $k+m\le n/2$ and $|t|\le\pi$, then by \eqref{ccgf},
or \eqref{johd},
\begin{equation}
  \label{solc}
\bigabs{\tgf^{n-k-m}(t)}
\le \exp\xpar{-\ccgf(n-k-m)t^2} \le \exp\xpar{-\ccjohd nt^2}. 
\end{equation}

Denote the \lhs{} of \eqref{sol} by $\gD_{k,m}$.
If we replace the factor 
$\tgf^m(t)e^{\ii t}-\tgf^m(u)e^{\ii u}$ in the \rhs{} by
$\tgf^m(t)\tgf^m(u)\bigpar{e^{\ii t}-e^{\ii u}}$, we obtain after
cancellation $\gD_{k,0}$.
Using \eqref{sola}--\eqref{solc} to estimate the resulting error we obtain
  \begin{equation}\label{sold}
	\begin{split}
&\bigabs{\gD_{k,m}-\gD_{k.0}}
\\&\quad
\ll
\intpipix\intpipix e^{-\ccjohd nt^2-\ccjohd nu^2}
\bigpar{|t|+|u|+kt^2+ku^2}
\bigpar{mt^2+mu^2}
\dd t\dd u
\\&\quad
\ll {m n^{-5/2}+km n^{-3}}.
	\end{split}
\raisetag{\baselineskip}
  \end{equation}
This is covered by the error terms in \eqref{lv2x}, and thus it suffices to
prove \eqref{lv2x} for $m=0$.
By symmetry, we also obtain $|\gD_{k,0}-\gD_{0,0}| \ll k n^{-5/2}$,
so it suffices to prove \eqref{lv2x} in the case $k=m=0$.

In that case, \eqref{sol} yields
  \begin{equation}\label{sol00}
	\begin{split}
\gD_{0,0}
&=
\frac{1}{8\pi^2}\intpipix\intpipix \tgf^{n}(t)e^{-\ii t}
 \tgf^{n}(u)e^{-\ii u} 
\bigpar{e^{\ii t}-e^{\ii u}}^2
\dd t\dd u
\\
&=
\frac{1}{8\pi^2 n^2}
\int_{-\pi\sqrt n}^{\pi\sqrt n}
\int_{-\pi\sqrt n}^{\pi\sqrt n}
 \tgf^{n}\Bigparfrac{x}{\sqrt n}e^{-\ii x/\sqrt n}
 \tgf^{n}\Bigparfrac{y}{\sqrt n}e^{-\ii y/\sqrt n} 
\times
\\& \hskip 14em
\bigpar{\sqrt n\bigpar{e^{\ii x/\sqrt n}-e^{\ii y/\sqrt n}}}^2
\dd x\dd y,
	\end{split}
  \end{equation}
and it follows by dominated convergence, using \eqref{tgf} and \eqref{ccgf},
that as \ntoo,
  \begin{equation}\label{sol000}
	\begin{split}
n^2 \gD_{0,0}
&\to
-\frac{1}{8\pi^2}
\intoooo
\intoooo
e^{-\gss x^2/2-\gss y^2/2}(x-y)^2
\dd x\dd y
=-\frac{4\pi}{8\pi^2\gs^4}.
	\end{split}
  \end{equation}
This shows \eqref{lv2x} in the special case $k=m=0$, and thus by the
estimate
\eqref{sold} for all $k$ and $m$ with $k+m\le n/2$, which completes the
proof
of \eqref{lv2x} and \eqref{lv2}.

\pfitem{ii}
To prove \eqref{lv2b}, we first observe that we may assume $m\le k$ by
symmetry. In this case $n \ge k \ge n/4$.
We again multiply by $\PP(S_n=n-1)^2$ and use \eqref{sol}.
We now use the estimate, by \eqref{ccgf}, for $|t|,|u|\le\pi$,
\begin{equation}\label{sola2}
  |\tgf^k(t)e^{\ii t}-\tgf^k(u)e^{\ii u}|
\le |\tgf^k(t)|+|\tgf^k(u)|
\le \exp\xpar{-\ccgf kt^2} + \exp\xpar{-\ccgf ku^2}. 
\end{equation}
Using this and \eqref{sola} (with $k$ replaced by $m$) in \eqref{sol} we
obtain, using $k\gg n$ and symmetry,
{\multlinegap=0pt
  \begin{multline*}
\gD_{k,m}
\ll
\intpipix\intpipix e^{-\ccgf\xpar{n-k-m}t^2-\ccgf\xpar{n-k-m}u^2}
\Bigpar{e^{-\ccgf kt^2}+e^{-\ccgf ku^2}}
\times
\\
\shoveright{
\bigpar{|t|+mt^2+|u|+mu^2)}
\dd t\dd u}
\\
\shoveleft{\phantom{\gD_{k,m}}
\le2e^{c_3\pi^2}
\intoooo\intoooo e^{-\ccgf k t^2-\ccgf\xpar{n-k-m+1}u^2}
\bigpar{|t|+mt^2+|u|+mu^2)}
\dd t\dd u}
\\
\shoveleft{\phantom{\gD_{k,m}}
\ll \frac{1}{n\qq(n-k-m+1)} + \frac{m}{n\qq(n-k-m+1)^{3/2}}.
\hfill}
  \end{multline*}}%
The result \eqref{lv2b} follows by dividing by $\PP(S_n=n-1)^2\gg n\qw$.
\end{proof}

In particular, we obtain a simple asymptotic result for fixed $k$ and $m$.

\begin{lemma}\label{Lcovkm}
For any fixed $k$ and $m$ with $k\ge m$, 
as \ntoo,
\begin{multline*}
\frac1n\Cov( F_k(\ctn), F_m(\ctn))
\to
\\
 \E\bigpar{f_k(\cT) F_m(\cT)}
-(k+m-1 + \gs^{-2})
\E{f_k(\cT)}\E f_m(\cT).
\end{multline*} 
\end{lemma}

\begin{proof}
  This now follows from \refL{Lcov}. After division by $n$, 
the two first terms on the \rhs{} of \eqref{lcov} converge to
$ \E\bigpar{f_k(\cT) F_m(\cT)}$ and $-(k+m-1)\E{f_k(\cT)}\E f_m(\cT)$
since the probability ratios converge to 1 by \refL{L2}(i).
The third term divided by $n$ is by \refL{Lv2}(i)
\begin{equation*}
\bigpar{n+O(1)}\bigpar{-\gs^{-2}n\qw+o(n\qw)}\E{f_k(\cT)}\E f_m(\cT)  
\to -\gs^{-2}\E{f_k(\cT)}\E f_m(\cT)  .
\end{equation*}
\end{proof}
  
This yields immediately  variance asymptotics for a functional $f$ with
finite support. 

\begin{corollary}\label{Cfinite}
  Suppose that $f$ has finite support.
Then, as \ntoo, 
{\multlinegap=0pt
\begin{multline*}
  \frac1n \Var F(\ctn)
\to
\\
\E\Bigpar{f(\cT)\bigpar{2F(\cT)-f(\cT)}}
-2\E\bigpar{|\cT|f(\cT)}\E f(\cT)
+\bigpar{1-\gs^{-2}}\bigpar{\E f(\cT)}^2.
\end{multline*}}%
\end{corollary}

\begin{proof}
By \eqref{Fsum} and \refL{Lcovkm}, the  limit exists and equals
\begin{multline*}
2\sum_k\sum_{m<k} \E\bigpar{f_k(\cT) F_m(\cT)}
+
\sum_{k} \E\bigpar{f_k(\cT) F_k(\cT)}
\\
-\sum_{k,m}
(k+m-1 + \gs^{-2})
\E{f_k(\cT)}\E f_m(\cT),
\end{multline*}
where all sums are finite since $f_k=0$ for large $k$.
Since $f_k(T)=0$ unless $|T|=k$, we have 
$f_k(T)F_k(T)=f_k(T)^2$ and $f_k(T)F_m(T)=0$ for $m>k$, 
$T\in\st$.
Using this, \eqref{Fsum} and the similar relations $\sum_k f_k(T)^2=f(T)^2$ and 
$\sum_k kf_k(T)=|T|f(T)$, it follows that
the limit can be written as
{\multlinegap=0pt
  \begin{multline*}
2\sum_k\sum_{m} \E\bigpar{f_k(\cT) F_m(\cT)}
-\sum_{k} \E\bigpar{f_k(\cT)^2}
-2\sum_k k\E f_k(\cT) \sum_m\E f_m(\cT)
\\\shoveright{
+\bigpar{1-\gs^{-2}} \biggpar{\sum_k \E f_k(\cT)}^2}
\\
= 2 \E\bigpar{f(\cT)F(\cT)}-\E\bigpar{f(\cT)^2}
-2 \E \bigpar{|T|f(\cT)} \E f(\cT)
+\bigpar{1-\gs^{-2}} \bigpar{\E f(\cT)}^2.
  \end{multline*}}%
\end{proof}

\begin{remark}
  \label{RCfinite}
The limit in \refC{Cfinite} equals $\gam^2$ in \eqref{gam} for every $f$
with finite support, and more generally for every $f$ such that the
expression in \eqref{gam} is finite and $\E\bigpar{|\cT||f(\cT)|}<\infty$.
Conversely, if $\E f(\cT)\neq0$,
the condition $\E\bigpar{|\cT||f(\cT)|}<\infty$
is necessary for the expression in
\refC{Cfinite} to be finite;
note that this condition by \eqref{pett} is equivalent to
$\sumni \E|f(\ctn)|/\sqrt n<\infty$ and thus imposes a stronger decay
of $\E f(\ctn)$ than \eqref{t1v2}.
\end{remark}

In order to extend this to more general functionals $f$, we prove a general
upper bound for the variance.
We first give another lemma estimating a combination of probability ratios
where there typically is a lot of cancellation.

\begin{lemma}
  \label{L3}
  \begin{xenumerate}
  \item 
If $m \le k/2$ and $k\le n$, then 
\begin{multline}\label{l3a}
k \frac{\PP\xpar{S_{k-m}=k-m}}{\PP(S_k=k-1)}
-\min\xpar{k+m-1,n}
\frac{\PP\xpar{S_{n-m}=n-m}}{\PP(S_n=n-1)}
\\
=O\bigpar{m}+O\bigpar{k\qq}.
\end{multline}
\item 
If $k/2 < m \le k \le n$, then 
\begin{multline}\label{l3b}
k \frac{\PP\xpar{S_{k-m}=k-m}}{\PP(S_k=k-1)}
-\min\xpar{k+m-1,n}
\frac{\PP\xpar{S_{n-m}=n-m}}{\PP(S_n=n-1)}
\\
=O\biggparfrac{k\qqc}{(k-m+1)\qq}.
\end{multline}
\item 
If furthermore $\E\xi^{2+\gd}<\infty$  with $0<\gd\le1$, then
the estimate in \textup{(i)} is improved to $O(m)+O\bigpar{k^{(1-\gd)/2}}$. 
\end{xenumerate}
\end{lemma}

\begin{proof}
  \pfitem{i}
By \refL{L2}(i) (twice), 
\begin{equation}\label{promo}
\begin{split}
&k \frac{\PP\xpar{S_{k-m}=k-m}}{\PP(S_k=k-1)}
-\xpar{k+m-1}
\frac{\PP\xpar{S_{n-m}=n-m}}{\PP(S_n=n-1)}
\\
&\quad=
k \Bigpar{1+O\Bigparfrac{m}{k}+O\bigpar{k\qqw}}
-(k+m-1)\Bigpar{1+O\Bigparfrac{m}{n}+O\bigpar{n\qqw}}	
\\
&\quad= O(m)+O\bigpar{k\qq},
\end{split}  
\raisetag{\baselineskip}
\end{equation}
which shows \eqref{l3a} if also $k+m-1\le n$.

If $k+m-1>n$, we have $0<k+m-1-n<m$ and thus, by \refL{L2}(i) again (or by
\eqref{llt}--\eqref{snn}),
\begin{equation}\label{enblad}
\xpar{k+m-1-n}
\frac{\PP\xpar{S_{n-m}=n-m}}{\PP(S_n=n-1)}
\\
= O(k+m-1-n) =O(m),
\end{equation}
and \eqref{l3a} follows by adding \eqref{promo} and \eqref{enblad}.

\pfitem{ii}
By \refL{L2}(ii),
\begin{multline*}
\lrabs{
k \frac{\PP\xpar{S_{k-m}=k-m}}{\PP(S_k=k-1)}
-\min\xpar{k+m-1,n}
\frac{\PP\xpar{S_{n-m}=n-m}}{\PP(S_n=n-1)}}
\\
\ll
k\frac{k\qq}{(k-m+1)\qq}
+ (k+m)\frac{n\qq}{(n-m+1)\qq},
\end{multline*}
yielding the result since $n/(n-m+1)\le k/(k-m+1)$.

\pfitem{iii}
We use the improved estimate in \refL{L23} instead of \refL{L2}(i) in
\eqref{promo} and the result follows as above.
\end{proof}

\CCreset
\begin{theorem}  \label{Tvar<}
For any functional $f:\st\to\bbR$,
\begin{equation}\label{tvar<1}
\Var\bigpar{F(\ctn)}\qq
\le \CC n\qq
\biggpar{
\sup_k\sqrt{\E f(\ctk)^2}
+\sumki \frac{\sqrt{\E f(\ctk)^2}}{k}} ,
\end{equation}
with $\CCx$ independent of $f$.
\end{theorem}

\begin{proof} 
Let $\mu_k := \E f(\cT_k) = \E f_k(\ctk)$.
By the decomposition $f(T)= f'(T)+f''(T)$ where
$f'(T):=f(T)-\mu_{|T|}$ and $f''(T):= \mu_{|T|}$, there is a
corresponding decomposition 
$F(\ctn)=F'(\ctn)+F''(\ctn)$. Minkowski's inequality $\Var(X+Y)\qq\le
\Var(X)\qq+\Var(Y)\qq$ (for any random variables $X$ and $Y$) shows that it
suffices to show the estimate for $F'(\ctn)$ and $F''(\ctn)$ separately.
In other words, it suffices to show \eqref{tvar<1}  in the two special cases
where either $\E f(\ctk)=0$ for every $k$ (so $f=f'$), 
or $f(T)=\mu_{|T|}$ depends on $|T|$ only (so $f=f''$).
Recall the notation $\pi_n:=\PP(|\cT|=n)$ from \eqref{pik}.

\pfcase{1}{$\E f(\ctk)=0$ for every $k$.}
In this case, 
since $f(\ctk)=f_k(\ctk)$,
also $\E f_k(\ctk)=0$ and
by \eqref{efkk2} and \eqref{lefkn} (and the trivial $F_k(\ctn)=0$ for $k>n$),
\begin{align}\label{ef0}
\E f_k(\cT)=0,
&&
\E F_k(\ctn)=0,
&&&
\E F_k(\cT)=\sumni \pi_n\E F_k(\ctn)=0,
\end{align}
for all $k\ge1$, $n\ge1$.
Hence,
\eqref{lcov} and \eqref{lcov2} yield the same result and
\refL{Lcov} reduces to, 
for  $m\le k\le n$,
\begin{equation}
  \begin{split}
\Cov( F_k(\ctn), F_m(\ctn))
&=n\frac{\PP\xpar{S_{n-k}=n-k}}{\PP(S_n=n-1)} \E\bigpar{f_k(\cT) F_m(\cT)}
\\
&=n\frac{\PP\xpar{S_{n-k}=n-k}}{\PP(S_n=n-1)}\pi_k \E\bigpar{f_k(\ctk) F_m(\ctk)}.
\end{split}
\end{equation}
Consequently,
using again $F_m(\ctk)=0$ for $m>k$ and $F_k(\ctk)=f_k(\ctk)$,
\begin{align}
\frac1n\Var F(\ctn)
&=
\frac1n\sum_{k=1}^n\sum_{m=1}^n\Cov( F_k(\ctn), F_m(\ctn)) \notag
\\
&=
\frac1n\sum_{k=1}^n\sum_{m=1}^k(2-\gd_{km})\Cov( F_k(\ctn), F_m(\ctn)) \nonumber
\\
&=
\sum_{k=1}^n\sum_{m=1}^k(2-\gd_{km})
\frac{\PP\xpar{S_{n-k}=n-k}}{\PP(S_n=n-1)}\pi_k \E\bigpar{f_k(\ctk) F_m(\ctk)}
\notag\\
&=
\sum_{k=1}^n
\frac{\PP\xpar{S_{n-k}=n-k}}{\PP(S_n=n-1)}\pi_k 
\E\bigpar{f_k(\ctk) \bigpar{2F(\ctk)-f_k(\ctk)}}
\label{emma1}\\
&\le
2\sum_{k=1}^n
\frac{\PP\xpar{S_{n-k}=n-k}}{\PP(S_n=n-1)}\pi_k 
\E\bigpar{f_k(\ctk) F(\ctk)}.
\label{emma2}
\end{align}
By \eqref{llt}--\eqref{snn}, \eqref{pett}, \eqref{ef0} and the \CSineq, 
this yields
\begin{equation*}
  \begin{split}
\frac1n\Var F(\ctn)
&\ll
\sum_{k=1}^n
\frac{n\qq}{(n-k+1)\qq k^{3/2}}\bigpar{\Var f_k(\ctk)}\qq 
\bigpar{\Var F(\ctk)}\qq.
\end{split}
\end{equation*}
Let us write $\Var f_k(\ctk)=\ga_k^2$ and $\Var F(\ctk)= k\gb_k^2$,
and let further 
\begin{equation}
  \label{B}
B:=\sup_k\ga_k+\sumki \frac{\ga_k}k
\end{equation}
and 
$\gbx_n:=\sup_{k\le n}\gb_k$. 
Then we have shown
\begin{equation}
  \begin{split}
\gb_n^2
&\ll
\sum_{k=1}^n
\frac{n\qq}{(n-k+1)\qq k^{3/2}}\ga_k k\qq\gb_k
\\&
\ll
 \sum_{k=1}^{n/2} \frac{\ga_k}k\gb_k
+ n\qqw \sum_{k=n/2}^{n} \frac{\ga_k}{(n-k+1)\qq}\gb_k
\\&
\ll
\gbx_n \sum_{k=1}^{\infty} \frac{\ga_k}k
+ n\qqw \gbx_n \sup_{k}\ga_k\sum_{k=n/2}^{n} \frac{1}{(n-k+1)\qq}
\\&
\ll
  B \gbx_n.
  \end{split}
\end{equation}
In other words, $\gb_n^2\le \CCx B\gbx_n$ for some $\CCx$.
The sequence $\gbx_n$ is increasing, and thus we obtain
\begin{equation}
  (\gbx_n)^2 = \sup_{1\le m\le n} \gb_m^2
\le \CCx B \gbx_n.
\end{equation} 
Consequently, recalling that $\gb_n$ and $\gbx_n$ are finite by
\refR{Rfinite},
\begin{equation}
  \gb_n \le \gbx_n \le \CCx B,
\end{equation}
\ie, $\Var F(\ctn)=n\gb_n^2 \le n\CCx^2 B^2$,
which, recalling \eqref{B}, completes the proof of Case 1.

\pfcase{2}{$f(T)=\mu_{|T|}$.}
In this case, $f(\ctk)=f_k(\ctk)=\mu_k$, 
and \eqref{ekm} and \eqref{lefkn} yield
\begin{equation*} 
  \begin{split}
\E \bigpar{f_k(\cT)F_m(\cT)}
&
=\PP(|\cT|=k)\E \bigpar{f_k(\ctk)F_m(\ctk)}
=\PP(|\cT|=k)\mu_k\E\bigpar{F_m(\ctk)}
\\&
=\E f_k(\cT)k \frac{\PP\xpar{S_{k-m}=k-m}}{\PP(S_k=k-1)} \E{f_m(\cT)}.
  \end{split}
\end{equation*}
Thus \refL{Lcov} can be written,
for $n\ge k\ge m$ 
(and assuming  $\PP(S_k=k-1)>0$;
otherwise $\PP(|\cT|=k)=0$ and $F_k(\cT_n)=0$ a.s., so we may ignore this case), 
{\multlinegap=0pt
\begin{multline}\label{cov2}
\frac1n\Cov( F_k(\ctn), F_m(\ctn))
\\\shoveleft{\quad
=\E{f_k(\cT)}\E f_m(\cT)
\biggl(
\frac{\PP\xpar{S_{n-k}=n-k}}{\PP(S_n=n-1)} \times}
\\\shoveright{
\biggpar{
k \frac{\PP\xpar{S_{k-m}=k-m}}{\PP(S_k=k-1)}
-\min\xpar{k+m-1,n}
\frac{\PP\xpar{S_{n-m}=n-m}}{\PP(S_n=n-1)}}}
\\\shoveleft{\quad\phantom{=\E{f_k(\cT)}\E f_m(\cT)\quad}
+(n-k-m+1)_+
\times}
\\\shoveright{
\quad\lrpar{\frac{\PP\xpar{S_{n-k-m}=n-k-m+1}}{\PP(S_n=n-1)} 
-\frac{\PP\xpar{S_{n-k}=n-k}}{\PP(S_n=n-1)} 
\frac{\PP\xpar{S_{n-m}=n-m}}{\PP(S_n=n-1)}}
\biggr)}
\\\shoveleft{\quad
=: (A_1+A_2)\E{f_k(\cT)}\E f_m(\cT).\hfill}
\end{multline}}%
We take absolute values and sum over all $m\le k\le n$
(the terms with $k>m$ are covered by symmetry). Cancellations inside $A_1$
and $A_2$ will be important, but  we treat the two terms $A_1$
and $A_2$ separately.

For convenience, we write
\begin{equation}
  x_k := |\mu_k|/\sqrt k.
\end{equation}
Thus, by \eqref{ek} and \eqref{pett}, 
\begin{equation}\label{blid}
  \E f_k(\cT) =\PP(|\cT|=k)\mu_k 
=O\bigpar{|\mu_k|/k\qqc}
=O(x_k/k).
\end{equation}
Consequently, by \eqref{Fsum}, \eqref{cov2} and symmetry, we have
\begin{equation}\label{jb}
  \frac1n\Var F(\ctn)
\ll \sum_{\substack{m\le k\\ k\le n}} (|A_1|+|A_2|) \frac{x_kx_m}{km}.
\end{equation}
To estimate this sum, we consider several cases.
We define
\begin{align}
\BB & := \sumki \frac{x_k}{\sqrt k} \BBdef\BBd 
= \sumki \frac{|\mu_k|}{k} 
,\\
\BB & := \sumki \frac{x_k}k \BBdef\BBa  
= \sumki \frac{|\mu_k|}{k\qqc} 
,\\
\BB &:= \sup_{n\ge1} \frac1{\sqrt n} \sumkn x_k \BBdef\BBb 
= \sup_{n\ge1} \frac1{\sqrt n} \sumkn \frac{|\mu_k|}{\sqrt k}
,\\
\BB & := \sup_n \sum_{n/2}^n \frac{x_k}{\sqrt{n-k+1}} \BBdef\BBc
\ll \sup_n \frac1{\sqrt n}\sum_{n/2}^n \frac{|\mu_k|}{\sqrt{n-k+1}} 
.\end{align}

\pfcasex{2}{$A_1$ for $m\le k/2$, $k\le n/2$.}
By \refL{L2}(i) (or \eqref{llt}--\eqref{snn})
and  \refL{L3}(i),
\begin{equation}\label{a1i}
|A_1| \ll { m+{k\qq}}.
\end{equation}
Hence, the contribution to \eqref{jb} from these terms is
\begin{equation}\label{sw1i}
  \begin{split}
\sum_{\substack{m\le k/2\\k\le n/2}}|A_1|\frac{x_kx_m}{km}
&\ll  \sum_{k\le n}\frac{x_k}{k}\sum_{m\le k}x_m 
+
\sum_{k\le n} \frac{x_k}{\sqrt k}
\sum_m \frac{x_m}m	
\\
&\le
\sum_{k\le n} \frac{x_k}{\sqrt k} (\BBb+\BBa)
\le \BBd(\BBb+\BBa).
  \end{split}
\end{equation}

\pfcasex{2}{$A_1$ for $m\le k/2$, $n/2 <k\le n$.}
By Lemmas \ref{L2}(ii)
and  \ref{L3}(i),
\begin{equation}\label{a1ii}
|A_1| \ll \frac{n\qq}{(n-k+1)\qq}\bigpar{ m+{k\qq}}
\le \frac{n\qq m}{(n-k+1)\qq}
+ \frac{n}{(n-k+1)\qq}.
\end{equation}
Hence, the contribution to \eqref{jb} from these terms is
\begin{equation}\label{sw1ii}
  \begin{split}
\sum_{\substack{m\le k/2\\n/2 <k\le n}}|A_1|\frac{x_kx_m}{km}
&\ll   \sum_{n/2<k\le n} \frac{x_k}{(n-k+1)\qq}
\lrpar{\sum_{m\le n}\frac{x_m}{n\qq} + \sum_{m} \frac{x_m}m}
\\  
&\le \BBc\bigpar{\BBb+\BBa}.	
  \end{split}
\end{equation}

\pfcasex{2}{$A_1$ for $k/2 < m\le k$, $k\le n/2$.}
By \refL{L2}(i) (or \eqref{llt}--\eqref{snn})
and  \refL{L3}(ii),
\begin{equation}\label{a1iii}
|A_1| 
\ll \frac{k\qqc}{(k-m+1)\qq}
\ll \frac{k\qq m}{(k-m+1)\qq}.
\end{equation}
Hence, the contribution to \eqref{jb} from these terms is
\begin{equation}\label{sw1iii}
\sum_{\substack{k/2<m\le k\\k\le n/2}}|A_1|\frac{x_kx_m}{km}
\ll 
\sum_{\substack{k/2<m\le k\\k\le n/2}}\frac{x_kx_m}{\sqrt k \sqrt{k-m+1}}
\le \BBc\sum_{k\le n/2}\frac{x_k}{\sqrt k}
\le \BBc\BBd.
\end{equation}

\pfcasex{2}{$A_1$ for $k/2 < m\le k$, $n/2 < k\le n$.}
By Lemmas \ref{L2}(ii)
and  \ref{L3}(ii),
noting $k \le n \ll m $,
\begin{equation} \label{a1iv}
|A_1| 
\ll \frac{n\qq}{(n-k+1)\qq}
\frac{k\qqc}{(k-m+1)\qq}
\ll \frac{km}{\sqrt{n-k+1}\sqrt{k-m+1}}.
\end{equation}
Hence, the contribution to \eqref{jb} from these terms is
\begin{equation}\label{sw1iv}
  \begin{split}
\sum_{\substack{k/2 <m\le k\\n/2 <k\le n}}|A_1|\frac{x_kx_m}{km}
\ll
\sum_{n/2<k\le n} \frac{x_k}{\sqrt{n-k+1}}
\sum_{k/2 <m\le k} \frac{x_m}{\sqrt{k-m+1}}
\le \BBc^2.	
  \end{split}
\raisetag{\baselineskip}
\end{equation}

\pfcasex{2}{$A_2$ for $m\le k$ and $m+k\le n/2$.}
In this case, \refL{Lv2}(i) yields 
\begin{equation}\label{a2i}
  |A_2| \ll(n-k-m+1)\Bigpar{\frac1n+\frac{k+m}{n^{3/2}}+\frac{km}{n^2}}
 \le 1+\frac{2k}{n^{1/2}}+\frac{km}{n}
\end{equation}
and  the contribution to  \eqref{jb} from these terms is dominated by
\begin{equation}\label{sw2i}
  \begin{split}
\sum_{k+m\le n/2}|A_2|\frac{x_kx_m}{km}
&\ll
\biggpar{ \sum_{k\le n} \frac{x_k}k}^2
+\frac{1}{n\qq} \sum_{k\le n} {x_k}
\sum_{m\le n} \frac{x_m}m
+\frac{1}n\biggpar{\sum_{k\le n}{x_k}}^2
\\&
\le \BBa^2+\BBb\BBa + \BBb^2.	
  \end{split}
\raisetag{\baselineskip}
\end{equation}

\pfcasex{2}{$A_2$ for $m\le k \le n$ and $m+k> n/2$.}
Note that $A_2$ vanishes unless $n\ge k+m$.
In this case, \refL{Lv2}(ii) yields 
\begin{equation}\label{a2ii}
|A_2|\ll
\frac{m\, n\qq}{(n-k-m+1)^{1/2}}
+{n\qq}.
\end{equation}
Since $k+m> n/2$ and $k\ge m$ imply $k> n/4$,
the contribution from these terms to \eqref{jb} is at most
\begin{equation}\label{sw2ii}
\begin{split}
   \sum_{k+m > n/2} |A_2|\frac{x_kx_m}{km}
&\ll
  \sum_{k= n/4}^n 
\frac{ x_k}{\sqrt n}
\sum_{m=1}^{n-k} 
\frac{ x_m}{(n-k-m+1)^{1/2}}
+  \sum_{k= n/4}^n 
\frac{ x_k}{\sqrt n}
\sum_{m=1}^{n-k} 
\frac{x_m}{m}
\\
&\ll 
\BBb(\BBb+\BBc)+\BBb\BBa.
\end{split}
\raisetag{\baselineskip}
\end{equation}

\pfcasexx{Conclusion} 
Consequently, \eqref{jb} together with
\eqref{sw1i}, \eqref{sw1ii}, \eqref{sw1iii}, \eqref{sw1iv},
\eqref{sw2i} and \eqref{sw2ii} show that 
\begin{equation}\label{jb+}
  \frac1n\Var F(\ctn)
\ll \bigpar{\BBd+\BBa+\BBb+\BBc}^2.
\end{equation}
Furthermore,
trivially $\BBa\le\BBd$ and
$\BBb,\BBc\ll \sup_k|\mu_k|$,
and $|\mu_k|\le \sqrt{\E f(\ctk)^2}$.
Hence, \eqref{jb+} proves
\eqref{tvar<1} in Case 2, which completes the proof.
\end{proof}

\begin{remark}
  The proof actually yields, noting that 
$\BBb\ll\BBc$,
  the slightly stronger
{\multlinegap=0pt 
\begin{multline*}
\Var\bigpar{F(\ctn)}\qq
\le \CC n\qq
\biggpar{
\sup_k \frac1{\sqrt k}\sum_{j=k/2}^k \frac{\sqrt{\E f(\cT_j)^2}}{\sqrt{k-j+1}}
+\sumki \frac{\sqrt{\E f_k(\ctk)^2}}{k}} .
  \end{multline*}}%
\end{remark}

\begin{remark}\label{Rbetter}
\refE{EJW} below shows that the term $\sum_k(\E f_k(\ctk)^2)\qq/k$ in
\eqref{tvar<1} cannot be improved in general. In the case $f(T)=\mu_{|T|}$,
there is, however,  a minor improvment
in the following
version of \refT{Tvar<},
provided we have more than a second moment of $\xi$.
This theorem implies, as mentioned in \refR{RT1}, 
a corresponding minor improvement of the condition
  \eqref{t1v2} in \refT{T1}; we omit the details.
(For example, it allows $f(T)=1/\log|T|$.)
We do not know whether $\Var F(\ctn)=O(n)$ for every bounded $f(T)$ that
depends on $|T|$ only.
\end{remark}

\begin{theorem}  \label{Tvar<2}
Suppose that $\E \xi^{2+\gd}<\infty$ with $\gd>0$ and 
let $\mu_k := \E f(\cT_k) = \E f_k(\ctk)$.
Then
{\multlinegap=0pt
\begin{multline*}
\Var\bigpar{F(\ctn)}\qq
\\\le \CC n\qq
\biggl(
\sup_k\bigpar{\E f_k(\ctk)^2}\qq
+\sumki \frac{(\Var f_k(\ctk))\qq}{k}	
+ 
\biggpar{\sumki \frac{\mu_k^2}{ k}}\qq   
\biggr).
  \end{multline*}}%
\end{theorem}

\begin{proof}
We use the notation of the proof of
 \refT{Tvar<}, and 
define $\BB$ by
\begin{align}
	\BBx^2 & := \sumki x_k^2   \BBdef\BBq  
=\sumki \frac{\mu_k^2}{k}.
  \end{align}
Observe that $\BBa\ll\BBq$ by the \CSineq.
We modify the proof of \refT{Tvar<}.  
Case 1 is as before, and so are Case 2(ii),(iv),(v),(vi),
leaving only two cases where we have to replace $\BBd$.

\pfcasexx{Case 2\,\textup{(i)}}
Using \refL{L3}(iii) instead of  \refL{L3}(i),
we obtain
\begin{equation}\label{a1ix}
|A_1| \ll  m+ k^{(1-\gd)/2}.
\end{equation}
Hence, the contribution to \eqref{jb} from these terms is
\begin{equation}\label{sw1ix}
  \begin{split}
\sum_{\substack{m\le k/2\\k\le n/2}}|A_1|\frac{x_kx_m}{km}
&\ll  \sum_{k,m}\frac{x_k x_m}{\max\xpar{k,m}}
+
\sum_{k} \frac{x_k}{k^{(1+\gd)/2}}
\sum_m \frac{x_m}m.	
  \end{split}
\end{equation}
The second term on the \rhs{} is $ \ll\sum_k x_k^2=\BBq^2$ by two
applications of the \CSineq.
The same holds for the first term, which
says that the infinite matrix 
$(1/\max\set{k,m})_{k,m=1}^\infty$ defines a bounded operator on $\ell^2$;
this follows from Hilbert's inequality
$\sum_{k,m} x_kx_m/(k+m) \le \pi\sum_kx_k^2$, 
see for example \cite[Chapter IX]{HLP}.

\pfcasexx{Case 2\,\upshape{(iii)}}
We use again \eqref{a1iii}, but in \eqref{sw1iii} we set $k=m+j$ and observe
that $j<k/2<m$ and thus, using the \CSineq,
{\multlinegap=0pt
\begin{multline*}
\sum_{\substack{k/2<m\le k\\k\le n/2}}\frac{x_kx_m}{\sqrt k \sqrt{k-m+1}}
\le \sum_{j=0}^\infty\frac{1}{\sqrt{j+1}} \sum_{m=j+1}^\infty
\frac{x_{m+j}x_m}{(m+j)^{1/4}m^{1/4} }
\\
\le \sum_{j=0}^\infty\frac{1}{\sqrt{j+1}}\sum_{m=j+1}^\infty
\frac{x_m^2}{m^{1/2}} 
=\sum_{m=1}^\infty\frac{x_m^2}{m^{1/2} } \sum_{j=0}^{m-1}\frac{1}{\sqrt{j+1}}
\ll 
\sum_{m=1}^\infty{x_m^2}=\BBq^2.		  
	\end{multline*}}

The different terms for Case 2 thus all are dominated by 
$(\BBa+\BBb+\BBc+\BBq)^2 \ll (\BBc+\BBq)^2$, which completes the proof.
\end{proof}

\begin{remark}
In \eqref{sw1ix} we used Hilbert's inequality.
One can also see in other ways that
$(1/\max\set{k,m})_{k,m=1}^\infty$ defines a bounded operator on $\ell^2$;
a much more general result 
is shown in \cite{Mazya} and \cite[Theorem 3.1]{SJ139} (for the continuous
case, which implies the discrete).
\end{remark}

In the special case of (weakly) local functionals, we can improve \refT{Tvar<}.
For simplicity we consider only bounded functionals.

\begin{theorem}\label{Tvarlocal}  \CCreset
Suppose that $f(T)$ is a bounded and weakly local functional on $\st$
with cut-off $M$.
Let
$A_k:=\sup\set{|f(T)|:|T|=k}$  
and $\mu_k := \E f(\cT_k)$. 
Then,
\begin{equation*}
\Var\bigpar{F(\ctn)}\qq
\le \CC n\qq
\biggpar{\Bigpar{\sup_k A_k\sup_k k^{-1/4}A_k}\qq
+\sup_k|\mu_k|
+\sumki \frac{|\mu_k|}{k}} 
\end{equation*}
where the constant $\CCx$ may depend on the cut-off $M$ but not otherwise
on $f$. 
\end{theorem}

\begin{proof} 
  We modify the proof of \refT{Tvar<}; Case 2 is the same so we consider
  Case 1 only. Thus assume that $\E f(\ctk)=0$ for each $k$.
We have
\begin{equation}
F(\ctk)
= \sum_{v\in\ctk}f(\cT_{k,v})
= \sum_{d(v)< M}f(\cT_{k,v})
+ \sum_{d(v)\ge M}f(\cT_{k,v})
=: S_1+S_2.
\end{equation}
Let again $w_j(T)$ be the number of nodes at depth $j$.
Since $|f|$ is bounded by $A$, the  sum $S_1$ is bounded by 
\begin{equation}\label{s1}
  |S_1| \le A\sum_{j=0}^{M-1} w_j(\ctk).
\end{equation}
As said in the proof of \refL{Lzw}, $\E w_j(\ctk) = O(j)$ for each $j$, 
see \cite{MM} and \cite[Theorem 1.13]{SJ167}.
Hence \eqref{s1} implies $\E|S_1| =O(A)$ and 
\begin{equation}
  \label{tib}
\E|f_k(\ctk)S_1|\le \E|A_kS_1| =O(A_kA).
\end{equation}
For $S_2$ we condition on $\ctk\MM$ and on the sizes $n_v$ of  
the $w_M(\ctk)$ subtrees $T_v$ with $d(v)=M$. 
Given this, the forest $\ctk\setminus\ctk\MMi$ consists of
independent copies of random trees $\cT_{d_v}$, and $S_2$ is the sum of
$f(\cT_{k,v})$ over all nodes in these trees, which equals the sum of
$F$ over these trees. Since each $\E F(\cT_{d_i})=0$ by \refL{Lefkn}, it
follows that the conditional expectation $\E\bigpar{S_2\mid \ctk\MM}=0$.
However, $f_k(\ctk)$ depends only on $\ctk\MM$, and thus
$\E \bigpar{f_k(\ctk) S_2}=0$. 
Consequently, \eqref{tib} yields
\begin{equation}\label{tibb}
\E \bigpar{f_k(\ctk)F(\ctk)}=O(A_k A) .
\end{equation}

By \eqref{emma2}, \eqref{tibb} and estimates using \eqref{llt}--\eqref{snn}
and \eqref{pett} as before,
\begin{equation*}
  \begin{split}
\frac1n\Var F(\ctn)
\ll \sum_{k=1}^{n/2} k^{-3/2}A_kA+\sup_{k>n/2}k\qqw A_kA, 	
  \end{split}
\end{equation*}
and the result follows. (The exponent $-1/4$ can be replaced by any exponent
$>-1/2$.) 
\end{proof}

\begin{example} \ccreset \label{EJW}
We show by an example, where we make correlations between different
$F_k(\ctn)$ large, that the condition \eqref{t1v2} in general cannot be
relaxed. We 
consider any offspring distribution such that $p_0,p_1,p_2>0$. 

Suppose that $(\ga_k)_3^\infty$ is a given sequence of
positive numbers. Define $f_1=f_2:=0$ and let
$g_3(T):=\#\set{\text{leaves in  $T$}}-c_0$ when $|T|=3$, 
where the constant $c_0$ is chosen such that
$\E g_3(\cT_3)=0$. (Note that $\cT_3$ has one or two leaves, 
both with positive probability,
so $g_3(\cT_3)\neq0$.) 
Let $f_3(T)=\au_3 g_3(T)$ for a constant $\au_3>0$ such that
$\Var f_3(\cT_3)=\ga_3^2$. 
Continue recursively. If we have chosen
$f_1,\dots,f_{k-1}$, let for a tree $T$ with $|T|=k$, 
\begin{equation}\label{gk}
g_k(T):=\sum_{j=1}^{k-1}F_j(T)=\sumy_{v\in  T} f_{|T_v|}(T_v),   
\end{equation}
where $\sumy$ denotes summation
over all nodes except the root.
Define 
\begin{equation}
  \label{fs}
f_k(T)=\au_k g_k(T), \qquad T\in\st_k, 
\end{equation}
for a constant $\au_k>0$ such that
$\Var f_k(\cT_k)=\ga_k^2$. 
Note that, by induction, and \refL{Lefkn},
$\E f_k(\cT_k)=\E g_k(\cT_k)=0$ for every $k$.

Consider $f=\sum_k f_k$ and the corresponding $F=\sum_k F_k$. 
By construction,
for  a tree $T$ with $|T|=k>3$,
\begin{equation}
F(T)=f_k(T)+g_k(T)=(1+\au_k)g_k(T).  
\end{equation}
If we let $\bv_k^2 := \Var g_k(\cT_k)$, we have $\ga_k^2=\au_k^2\bv_k^2$ 
so $\ga_k=\au_k\bv_k$. Thus, for $k>3$,
\begin{equation}
  \label{bl2}
\Var F(\cT_k) =(1+\au_k)^2\Var g_k(\cT_k) = 
(1+\au_k)^2\bv_k^2=
(\ga_k+\bv_k)^2
\end{equation}
and
\begin{equation}\label{bl}
  \begin{split}
  \E\bigpar{f_k(\cT_k)\bigpar{2F(\cT_k)-f_k(\cT_k)}}
&=
  \E\bigpar{\au_k g_k(\cT_k)(2+\au_k)g_k(\cT_k)}
\\&
=\au_k(2+\au_k) \bv_k^2
=\ga_k(2\bv_k+\ga_k).	
  \end{split}
\end{equation}

For $k=3$, $F(\cT_3)=f_3(\cT_3)$, and \eqref{bl2}--\eqref{bl} hold if
we redefine $\bv_3:=0$.

By \eqref{gk}, \eqref{emma1} (with $f_n$ temporarily redefined as 0)
and \eqref{bl},
\begin{equation}\label{per}
  \begin{split}
  \bv_n^2 &
=\Var g_n(\cT_n) 
=n\sum_{k=3}^{n-1} 
\frac{\PP\xpar{S_{n-k}=n-k}}{\PP(S_n=n-1)}\pi_k 
\E\bigpar{f_k(\ctk) \bigpar{2F(\ctk)-f_k(\ctk)}}
\\
&=n\sum_{k=3}^{n-1} 
\frac{\PP\xpar{S_{n-k}=n-k}}{\PP(S_n=n-1)}\pi_k 
\ga_k(2\bv_k+\ga_k).
  \end{split}
\raisetag{\baselineskip}
\end{equation}
In particular, for $n>3$,
using \eqref{llt} (or \refL{L2}),
\begin{equation}\label{magnus}
  \bv_n^2 
\ge n
\frac{\PP\xpar{S_{n-3}=n-3}}{\PP(S_n=n-1)}\pi_3
\ga_3^2
\ge \cc n.
\end{equation}
If $n-k\ge1$, \eqref{llt} implies 
$\PP(S_{n-k}=n-k) \ge \cc (n-k)\qqw \ge \ccx n\qqw$
and thus, using also \eqref{snn},
$\PP(S_{n-k}=n-k)/\PP(S_{n-1}=n) \ge \cc$.
Using this, \eqref{magnus} and \eqref{pett} in 
\eqref{per} yields,
noting \eqref{bl2},
\begin{equation}
  \begin{split}
\Var F(\ctn)\ge
  \bv_n^2 &
\ge \cc
n\sum_{k=4}^{n-1} 
k^{-3/2}
\ga_k k\qq.
  \end{split}
\end{equation}
It follows that if 
$\sum_{k=3}^\infty \xfrac{\ga_k}{k}=\infty$,
then $\Var F(\ctn)/n\to\infty$.

Consequently, the condition \eqref{t1v2} is in general necessary for 
$\Var F(\ctn) =O(n)$, even if we assume $\E f(\ctn)=0$.
In particular, taking $\ga_k=1/\log k$, $k\ge3$, we find an example where
$\E f(\ctk)^2\to0$ as $k\to0$ but $\Var F(\ctn)/n\to\infty$.
\end{example}

\begin{example}
  \label{ESW} 
We modify example \refE{EJW} to have $f_k(T)$ uniformly bounded by defining
recursively, instead of \eqref{fs}, 
\begin{equation}
  \label{fs+}
f_k(T)=\au_k \bigpar{\sign(g_k(T))-\E\sign(g_k(\ctk))}, \qquad T\in\st_k,
\end{equation}
for a given bounded sequence $(\au_k)_3^\infty$ of positive numbers.
We still have $\E f_k(\ctk)=\E g_k(\ctk)=0$. Furthermore,
\begin{equation}
  \E \bigpar{f_k(\ctk)g_k(\ctk)}
= \au_k\E |g_k(\ctk)|.
\end{equation}
Since $F(\ctn)=f_n(\ctn)+g_n(\ctn)$,
it follows, similarly to \eqref{per}, that
\begin{equation}\label{axel}
  \begin{split}
\Var F(\ctn)
&\ge \Var g_n(\ctn) 
\\&
=n\sum_{k=3}^{n-1} 
\frac{\PP\xpar{S_{n-k}=n-k}}{\PP(S_n=n-1)}\pi_k 
\E\bigpar{f_k(\ctk) \bigpar{2F(\ctk)-f_k(\ctk)}}
\\
&\ge \cc n\sum_{k=3}^{n-1} 
\pi_k \au_k\E |g_k(\ctk)|
\ge \cc n\sum_{k=3}^{n-1} 
 \au_k\E |g_k(\ctk)|/k^{3/2}.
  \end{split}
\raisetag{\baselineskip}
\end{equation}
In particular $\Var g_n(\ctn)\ge \cc n$. It seems likely that also
\begin{equation}\label{aa}
\E|g_n(\ctn)|\ge \cc n\qq;   
\end{equation}
if this is the case with, for example, $\au_k=1/\log k$,
then \eqref{axel} shows that $\Var F(\ctn)/n\to\infty$ although $f$ is bounded,
\eqref{t1v1} holds
and $\E f(\ctn)=0$ for all $n$.

Unfortunately, we have not been able to show \eqref{aa}, but we note that 
if \eqref{t1n} holds (with $\mu=0$ as in our case), then 
$\liminf_\ntoo \E|F(\ctn)|/\sqrt n \ge \sqrt{2/\pi}\gam$ by Fatou's lemma, 
and since
$F(\ctn)-g_n(\ctn)=f_n(\ctn)=O(1)$, it follows that \eqref{aa} holds.
Hence we can at least conclude that, for $\au_k=1/\log k$, 
\eqref{t1v} and \eqref{t1n} cannot both hold (for any finite $\gam^2$).
\end{example}

\section{Asymptotic normality}\label{San}

In this section we consider only functionals $f$ with finite support.
Recall that this implies that $f$ is bounded.

\begin{lemma}\label{Lan}
  Suppose that $f$ has finite support.
Then, with notations as in \refT{T1},
\begin{equation}\label{lan}
\frac{ F(\ctn) - n\mu}{\sqrt n} \dto N(0,\gam^2).
\end{equation}
\end{lemma}

\begin{proof}
  We use the representation \eqref{Fd}. Since $f$ has finite support, there
  exists $m$ such that $f_k=0$ for $k>m$; this means that it suffices to sum
  over $k\le m$ in \eqref{Fd}.
We define
\begin{equation}\label{g8}
  g(x_1,\dots,x_m):=\sum_{k=1}^m f(x_1,\dots,x_k)
=\sum_{k=1}^m f_k(x_1,\dots,x_k).
\end{equation}
Then \eqref{Fd} can be written (assuming $n\ge m$)
\begin{equation}\label{Fg}
  F(\ctn)
\eqd \lrpar{\sum_{i=1}^{n} g(\xi_i,\dots,\xi_{i+m-1\bmod n})
\Bigm| S_n=n-1}.
\end{equation}
We now use a method by \citet{LeCam} and \citet{Holst81},
see also \citet{Kudlaev}. (See in particular
\cite[Theorem 5.1]{Holst81}; the conditions are somewhat different but we use
essentially the same proof.)

Note first that by \eqref{g8} and \eqref{efkkt},
\begin{equation}\label{eg}
\E g(\xi_1,\dots,\xi_m)=\sum_{k=1}^m \E f_k(\cT)=\E f(\cT)=\mu.  
\end{equation}
Furthermore, $g$ is bounded, because $f$ is.

Fix $\ga$ with $0<\ga<1$ and a sequence $n'=n'(n)$ with $n'/n\to \ga$, for
example $n'=\floor{\ga n}$. Define the centred sum
\begin{equation}\label{yn}
  Y_n := \sumin \bigpar{g(\xi_i,\dots,\xi_{i+m-1})-\mu}.
\end{equation}
Then, by the standard central limit for $m$-dependent variables
\cite{HoeffdingR}, \cite{Diananda}, applied to the
random vectors $\bigpar{g(\xi_i,\dots,\xi_{i+m-1})-\mu,\xi_i}$,
\begin{equation}\label{ylim}
  \lrpar{\frac{Y_{n'}}{\sqrt{n}},\frac{S_{n'}-n'}{\sqrt n}}
\dto N\lrpar{0,\ga
  \begin{pmatrix}
\gb^2 & \rho \\ \rho & \gss	
  \end{pmatrix}
}
\end{equation}
where 
\begin{align*}
  \gb^2 &= 
\Var\bigpar{g(\xi_1,\dots,\xi_{m})}
+ 2\sum_{i=2}^{m}
\Cov\bigpar{g(\xi_1,\dots,\xi_{m}),g(\xi_i,\dots,\xi_{i+m-1})},
\\
\rho &= \sumim \Cov\bigpar{g(\xi_1,\dots,\xi_{m}),\xi_i}
=\Cov\bigpar{g(\xi_1,\dots,\xi_{m}),S_m}.
\end{align*}
We calculate $\gb^2$ by expanding $g$ using \eqref{g8} and arguing as 
in \eqref{vfn}---\eqref{vfnx} in the proof of \refL{Lcov} (where
we condition on $S_{n}=n-1$, making the present calculation simpler).
This yields, omitting the details, \cf{} also the proof of \refC{Cfinite},
and using \eqref{gam},
\begin{equation}\label{gbx}
  \begin{split}
\gb^2 &= \sum_{\ell\le k}^m (2-\gd_{k,\ell})\E\bigpar{f_k(\cT)F_\ell(\cT)}	
-\sum_{k,\ell=1}^m (k+\ell-1)\E f_k(\cT)\E f_\ell(\cT)
\\
&= \E\bigpar{f(\cT)(2F(\cT)-f(\cT))}- 2 \E\bigpar{|\cT|f(\cT)} \mu + \mu^2
\\
&= \gam^2 + \mu^2/\gss.
  \end{split}
\raisetag{\baselineskip}
\end{equation}
Furthermore, since $f_k(\xi_1,\dots,\xi_k)\neq0$ only when
$(\xi_1,\dots,\xi_k)$ is the degree sequence of a tree, and thus $S_k=k-1$,
while $\E S_k=k$, and using \eqref{efkkt} again,
\begin{equation}\label{rhox}
  \begin{split}
\rho &= \sum_{k=1}^m \Cov\bigpar{f_k(\xi_1,\dots,\xi_k),S_m}	
 = \sum_{k=1}^m \Cov\bigpar{f_k(\xi_1,\dots,\xi_k),S_k}	
\\&
 = \sum_{k=1}^m \E\bigpar{f_k(\xi_1,\dots,\xi_k)(S_k-k)}	
 = \sum_{k=1}^m \E\bigpar{-f_k(\xi_1,\dots,\xi_k)}	
\\&
=-\sum_{k=1}^m \E f_k(\cT)
=-\mu.
  \end{split}
\end{equation}
We define for convenience 
\begin{equation}\label{tY}
\tY_n := Y_n +\frac{\mu}{\gss}(S_n-n).  
\end{equation}
Then \eqref{ylim} yields, using \eqref{gbx}--\eqref{rhox},
\begin{equation}\label{ylim2}
  \lrpar{\frac{\tY_{n'}}{\sqrt{n}},\frac{S_{n'}-n'}{\sqrt n}}
\dto N\lrpar{0,\ga
  \begin{pmatrix}
\gb^2-\mu^2/\gss & 0 \\ 0 & \gss	
  \end{pmatrix}
}
=
 N\lrpar{0,\ga
  \begin{pmatrix}
\gam^2 & 0 \\ 0 & \gss	
  \end{pmatrix}
}.
\end{equation}
In other words, $\tY_{n'}/\sqrt{n}$ and $(S_{n'}-n')/\sqrt{n}$ are
jointly asymptotically normal with independent limits $W\sim N(0,\ga\gam^2)$
and $Z\sim N(0,\ga\gss)$.

Next, let $h$ be any bounded continuous function on $\bbR$. Then,
using \eqref{llt}--\eqref{snn} and \eqref{ylim2},
\begin{equation*}
  \begin{split}
&  \E \bigpar{h\bigpar{\tY_{n'}/\sqrt n}\mid S_n=n-1}
\\&=
\frac{\E\sum_j h\bigpar{\tY_{n'}/\sqrt n}\ett{S_{n'}=j}\ett{S_n-S_{n'}=n-1-j}}
{\PP(S_n=n-1)}	
\\&=
\frac{\sum_j \E\bigpar{h\bigpar{\tY_{n'}/\sqrt n}\ett{S_{n'}=j}}
\PP\bigpar{S_{n-n'}=n-1-j}}
{\PP(S_n=n-1)}	
\\&=
\sum_j \E\bigpar{h\bigpar{\tY_{n'}/\sqrt n}\ett{S_{n'}=j}}
\Bigparfrac{n}{n-n'}\qq \bigpar{e^{-(j+1-n')^2/(2(n-n')\gss)}+o(1)}
\\&=
(1-\ga)\qqw \E\Bigpar{h\bigpar{\tY_{n'}/\sqrt n}
e^{-(S_{n'}+1-n')^2/(2(n-n')\gss)}}
+o(1)
\\& \to
(1-\ga)\qqw\E\Bigpar{h(W)e^{-Z^2/(2(1-\ga)\gss)}}
\\&
=
\E\bigpar{h(W)}
(1-\ga)\qqw\E e^{-Z^2/(2(1-\ga)\gss)}
= \E h(W),
  \end{split}
\end{equation*}
where the final equality follows by a simple calculation, or even more
simply by using the special case $h\equiv1$.
Since $h$ is arbitrary, this proves
\begin{equation}\label{ylim3}
  \Bigpar{\tY_{n'}/\sqrt n\mid S_n=n-1} \dto W \sim N\bigpar{0,\ga\gam^2}.
\end{equation}

Next, conditioned on $S_n=n-1$ we have by 
\eqref{yn}, \eqref{tY} and symmetry (for $n$ so large that $n',n-n'>m$)
\begin{equation}\label{quf}
  \begin{split}
&
\sum_{i=1}^{n} \bigpar{g(\xi_i,\dots,\xi_{i+m-1\bmod n})-\mu}
-\tY_{n'}   
\\&\qquad
\eqd
\sum_{i=1}^{n-n'} \bigpar{g(\xi_i,\dots,\xi_{i+m-1})-\mu}
+\frac\mu{\gss}\bigpar{S_{n-n'}-(n-1-n')}
\\&\qquad
=\tY_{n-n'}+\mu/\gss.	
  \end{split}
\raisetag{\baselineskip}
\end{equation}
We may for notational
convenience pretend that the equality in distribution \eqref{Fd}
is an equality, and we then have, for each $\ga\in(0,1)$, a decomposition 
\begin{equation}\label{decomp}
  \frac{F(\ctn)-n\mu}{\sqrt n} = X'_{n,\ga}+X''_{n,\ga}
\end{equation}
where $X'_{n,\ga} = \bigpar{\tY_{n'}/\sqrt n\mid S_n=n-1}$ and, 
by \eqref{quf},
\begin{equation}
X''_{n,\ga} 
\eqd \biggpar{\frac{\tY_{n-n'}}{\sqrt n}\mid S_n=n-1}+\frac{\mu}{\gss\sqrt n}  
=  \biggpar{\frac{\tY_{n-n'}}{\sqrt n}\mid S_n=n-1}+o(1).
\end{equation}
By \eqref{ylim3}, $X'_{n,\ga}\dto W'_\ga\sim N\bigpar{0,\ga\gam^2}$.
Furthermore, $(n-n')/n\to1-\ga$, and thus by \eqref{ylim3} applied to
$1-\ga$ instead of $\ga$,
$X''_{n,\ga}\dto W''_\ga\sim N\bigpar{0,(1-\ga)\gam^2}$.

Now let $\ga\to1$ (along a sequence, if you like). 
Then $W'_\ga\dto N(0,\gam^2)$ and $W''_\ga\pto 0$, and 
the conclusion \eqref{lan} follows from
\eqref{decomp},
see \eg{} \cite[Theorem 4.2]{Billingsley} 
or \cite[Theorem 4.28]{Kallenberg}.
\end{proof}

\section{Final proofs}\label{Spf}

\begin{proof}[Proof of \refT{T1}]
We have already proved part \ref{T1e} in \refS{Sexp}.

Futhermore, we have proved part \ref{T1v} in the special case of a
functional $f$ with finite support in \refC{Cfinite} and \refL{Lan}.
In general, we use a truncation. 

We begin by verifying that $\gam^2$ is finite, with
the expectations in \eqref{gam} absolutely convergent.

First, $\E|f(\cT)|<\infty$ 
by assumption, see also  \refR{RT1a}.
Similarly, by \eqref{pett}, since $\E f(\ctn)^2=O(1)$ by \eqref{t1v1},
\begin{equation}\label{chr}
  \E f(\cT)^2 = \sumni \pi_n \E f(\ctn)^2
\ll \sumni \frac{\E f(\ctn)^2 }{n^{3/2}}
<\infty.
\end{equation}
Hence $\Var f(\cT)<\infty$.

To show that
$\E\bigabs{f(\cT)\bigpar{F(\cT)-|\cT|\mu}}<\infty$, 
note that by \refT{Tvar<} and \eqref{t1e},
\begin{equation}\label{oslo}
  \begin{split}
\E\bigpar{F(\ctn)-n\mu}^2	
= \Var F(\ctn) + \bigpar{\E F(\ctn)-n\mu}^2	
=O(n).
  \end{split}
\end{equation}
Thus, using the \CSineq, \eqref{pett}, \eqref{oslo} and  \eqref{t1v2},
\begin{equation}\label{hw}
  \begin{split}
  \E\bigabs{f(\cT)\bigpar{F(\cT)-|\cT|\mu}}	
&=
\sumni \pi_n \E\bigabs{f(\ctn)\bigpar{F(\ctn)-n\mu}}
\\&
\le
\sumni \pi_n \sqrt{\E {f(\ctn)^2}}\sqrt{\E \bigpar{F(\ctn)-n\mu}^2}
\\&
\ll
\sumni n\qqcw \sqrt{\E {f(\ctn)^2}} \,n\qq
<\infty.
  \end{split}
\end{equation}
Hence, $\gam^2$ is well-defined by \eqref{gam}, and finite.

Define the truncation 
\begin{equation}
f\NN(T):=\sum_{k=1}^N f_k(T) = f(T)\ett{|T|\le N}  
\end{equation}
and the corresponding sum $F\NN(T)$.
Furthermore, let $\mu\NN:=\E f\NN(\cT)$ and  
\begin{equation}\label{gamN}
(\gam\NN)^2 := 2\E\Bigpar{f\NN(\cT)\bigpar{F\NN(\cT)-|\cT|\mu\NN}}
-\Var f\NN(\cT) -(\mu\NN)^2/\gss.
\end{equation}
Then $\mu\NN\to\mu$ as \Ntoo{} by dominated convergence, and similarly,
using \eqref{chr},
$\Var f\NN(\cT)\to \Var f(\cT)$ and, using \eqref{hw},
\begin{equation*}
 \E\bigabs{f\NN(\cT)\bigpar{F\NN(\cT)-|\cT|\mu}}	
\to
  \E\bigabs{f(\cT)\bigpar{F(\cT)-|\cT|\mu}}	.
\end{equation*}
Finally, using \eqref{t1v1} and \eqref{pett},
\begin{equation*}
  \begin{split}
&  \E\bigabs{f\NN(\cT)|\cT|(\mu-\mu\NN)}
=
  \E\bigabs{f\NN(\cT)|\cT|}\cdot\bigabs{\mu-\mu\NN}
\\
&\quad\le
\sum_{k=1}^N \pi_k k \E|f(\ctk)| \cdot
\sum_{k=N+1}^\infty \pi_k |\E f(\ctk)|
=O\bigpar{N\qq}\cdot o\bigpar{N\qqw}
=o(1),
  \end{split}
\end{equation*}
as \Ntoo.
Combining these estimates we see that $(\gam\NN)^2\to\gam^2$.

Since $f\NN$ has finite support, \refC{Cfinite} yields
$\Var F\NN(\ctn)/n\to \gamN^2$ as $\ntoo$, for every fixed $N$.
Furthermore,
\refT{Tvar<} applied to $f-f\NN=\sum_{k=N+1}^\infty f_k$ shows that
\begin{equation}\label{rest}
n\qqw  \Var\bigpar{F(\ctn)-F\NN(\ctn)}\qq 
\ll \sup_{k>N}\sqrt{\E f(\ctk)^2}
+\sum_{k=N+1}^\infty \frac{\sqrt{\E f(\ctk)^2}}{k},
\end{equation}
uniformly in $n$ and $N$.
The \rhs{} is independent of $n$ and tends to 0 as \Ntoo, and it follows by 
Minkowski's inequality and a standard $3\eps$-argument
(\ie, because a uniform limit of convergent sequences is convergent)
that $n\qqw \Var (F(\ctn))\qq\to \lim_{N\to\infty}\gamNx=\gam$, showing
\eqref{t1v}.

Similarly, \refL{Lan} applies to each $f\NN$, and 
the uniform estimate \eqref{rest} implies that
we can let \Ntoo{} and conclude \eqref{t1n}, 
see \eg{} \cite[Theorem  4.2]{Billingsley} or \cite[Theorem 4.28]{Kallenberg}
again.   
\end{proof}

\begin{proof}[Proof of \refT{Tlocal}] 
Suppose first that $f$ is bounded and local. By replacing $f(T)$ by 
$f(T)-\E f(\hct)$, which does not change $F(T)-|T|\mu$, we may assume that
$\E f(\hct)=0$. In this case \refL{Lfloc} yields $\E f(\ctn)=O\bigpar{n\qqw}$
and in particular $\E f(\ctn)\to0$ and $\sum_n|\E f(\ctn)|/n<\infty$.
Hence the conditions of the second part are satisfied, so it suffices to
prove it.

Hence, assume now that $f$ is bounded and weakly local, 
and that $\E f(\ctn)\to0$ and $\sum_n|\E f(\ctn)|/n<\infty$.
We use truncation as in the proof of \refT{T1} and note first that
\refT{Tvarlocal} applied to $f-f\NN$ yields
\begin{equation}\label{hede}
  n\qqw\Var\bigpar{F(\ctn)-F\NN(\ctn)}\qq
\ll N^{-1/8}\sup _{k> N} A_k + \sup_{k>N}|\mu_k| + \sum_{k>N}|\mu_k|/k
\end{equation}
 where the \rhs{} is independent of $n$ and tends to 0 as $\Ntoo$.
(Note that $f-f\NN$ is weakly local with the same cut-off as $f$ for all
 $N$.)

Similarly, if $M\ge N$, then \refC{Cfinite} and \refT{Tvarlocal},
together with Minkowski's inequality,  show, 
with $\gamNx$ given by \eqref{gamN},
\begin{equation}
  \begin{split}
\bigabs{\gam\MM-\gam\NN} 
&\le \limsup_\ntoo n\qqw\Var\bigpar{F\MM(\ctn)-F\NN(\ctn)}\qq
\\&
\ll N^{-1/8}\sup _{k> N} A_k + \sup_{k>N}|\mu_k| + \sum_{k>N}|\mu_k|/k.	
  \end{split}
\end{equation}
Consequently $\bigpar{\gam\NN}_N$ is a Cauchy sequence so
$\gamNx \to\gam$
for some $\gam<\infty$.

The rest of the proof is the same as for \refT{T1}.
\end{proof}

\begin{proof}[Proof of \refC{C1}]
For any finite set $T_1,\dots,T_m$ of distinct trees
and real numbers $a_1,\dots,a_m$,
apply \refT{T1} to $f(T):=\sumim a_i\ett{T=T_i}$ and note that then 
$F(T)=\sumim a_i n_{T_i}(T)$. The assumptions \eqref{t1v1}--\eqref{t1v2}
hold trivially since $f$ has finite support.
We have $\mu = \E f(\cT)=\sumin a_i\PP(\cT=T_i)=\sumim a_i\pi_{T_i}$ and
a simple calculation shows that \eqref{gam} yields
$\gam^2 = \sum_{i,j=1}^m a_ia_j\gam_{T_i,T_j}$.
The results now follow from \refT{T1} 
(or directly from \refC{Cfinite} and \refL{Lan}),
using the Cram\'er--Wold device.
\end{proof}

\begin{proof}[Proof of \refT{Tgam}]
We use the notation in the proof of \refL{Lan}
(but now simply taking $n'=n$ so $\ga=1$).
We showed in \eqref{ylim} and \eqref{ylim2} asymptotic normality of $Y_n$,
$S_n$ and $\tY_n$; a simple (and well-known) calculation shows that also the
(co)variances converge:
$\Var(Y_n)/n\to\gb^2$, $\Cov(Y_n,S_n)/n\to\rho$, $\Var(S_n)/n\to\gss$ and,
recalling \eqref{gbx}--\eqref{rhox},
\begin{equation}\label{varty0}
  \frac{\Var(\tY_n)}{n} \to \gb^2 - \frac{\mu^2}{\gss} = \gam^2 =0.
\end{equation}

However, by \eqref{tY} and \eqref{yn},
\begin{equation} \label{tyg}
  \tY_n := \sumin \tg(\xi_i,\dots,\xi_{i+m-1}),
\end{equation}
where
\begin{equation}\label{tg}
\tg(\xi_i,\dots,\xi_{i+m-1})
:=g(\xi_i,\dots,\xi_{i+m-1})-\mu + \frac{\mu}{\gss}(\xi_i-1).
\end{equation}
The sequence $X_i:=\tg(\xi_i,\dots,\xi_{i+m-1})$ is strictly stationary and
$(m-1)$-dependent, with mean $\E X_i=0$ (by \eqref{eg}) and finite variance.
If the partial sums $\tY_n$ satisfy \eqref{varty0}, with limit $\gam^2=0$,
then as a consequence of a theorem by \citet{Leonov}, in the version given
by \citet[Theorem 8.6]{Bradley},
see \citet[Theorem 1.6]{SJmdep} for details,
there exists a function $h:\bbN^{m-1}\to\bbR$ such that
\begin{equation}
\tg(\xi_i,\dots,\xi_{i+m-1})=
h(\xi_{i+1},\dots,\xi_{i+m-1}) - h(\xi_i,\dots,\xi_{i+m-2}) \xas
\end{equation}
and thus by \eqref{tyg},
\begin{equation}\label{hh}
  \tY_n =
h(\xi_{n+1},\dots,\xi_{n+m-1}) - h(\xi_1,\dots,\xi_{m-1}) \xas
\end{equation}
In particular, $\tY_n$ depends \as{} only on $\xi_1,\dots,\xi_{m-1}$ and
$\xi_{n+1},\dots,\xi_{n+m-1}$, but not on $n$ or $\xi_{m},\dots,\xi_n$.

Consider now first case (i). Take $j>0$ with $p_j>0$ and consider the case
$\xi_i=j$ for all $i <n+m$. Then no substring of $\xi_1,\dots,\xi_{n+m-1}$
is the \ds{} of a tree. Thus, recalling \eqref{g8},
$g(\xi_i,\dots,\xi_{i+m-1})=0$ for every $i$ so \eqref{tg} and \eqref{tyg}
yield 
$\tg(\xi_i,\dots,\xi_{i+m-1})=\mu\bigpar{(j-1)/\gss-1}$ 
and
\begin{equation}
  \tY_n = n\mu\bigpar{(j-1)/\gss-1}.
\end{equation}
Since this vanishes for any $n$ by \eqref{hh},
$\mu\bigpar{(j-1)/\gss-1}=0$. By assumption, there exist at least two
different such $j$, and thus $\mu=0$.
Hence, \eqref{tg} simplifies to
$ \tg(\xi_i,\dots,\xi_{i+m-1})= g(\xi_i,\dots,\xi_{i+m-1})$,
and thus
\begin{equation} \label{tygg}
  \tY_n = \sumin g(\xi_i,\dots,\xi_{i+m-1}).
\end{equation}

Next consider the case $\xi_i=0$ for all $i<n+m$. Since $(0)$ is the \ds{}
of the tree $\bullet$ of size 1, \eqref{g8} yields
$g(0,\dots,0)=f_1(0)=f(\bullet)$. Hence, \eqref{tygg} yields 
$\tY_n=nf(\bullet)$. Since this vanishes, by \eqref{hh} again,
we must have $f(\bullet)=0$.

Suppose, in order to obtain a contradiction, that $f(\ct)$ does not vanish
a.s. We have, for some $N\ge1$, some
distinct trees $T_1,\dots,T_N$ and some real numbers
$a_1,\dots,a_N$,
\begin{equation}\label{fta}
  f(T)=\sum_{i=1}^N a_in_{T_i}(T),
\end{equation}
where we may assume that $a_i\neq0$ and $\PP(\ct=T_i)>0$ for all $i$
(otherwise we eliminate the offending terms). We may also suppose that
$T_1,\dots,T_N$ are ordered with $|T_1|\le |T_2|\le\dots$; this implies that
$f(T)=0$ for every proper subtree $T$ of $T_1$.
Let $T_1$ have \ds{} $(d_1,\dots,d_\ell)$, and 
consider now the case 
$\xi_{m+j}=d_j$, $j=1,\dots,\ell$, and $\xi_i=0$ for $i\le m$ and
$m+\ell<i<n+m$,
for $n\ge m+\ell$.
Since $f(T)=0$ for all proper subtrees of $T_1$ and $f(0)=0$, the only
non-zero contribution to $\tY_n$ is by \eqref{tygg}, \eqref{g8} and 
\eqref{fta},
\begin{equation}
f(\xi_{m+1},\dots,\xi_{m+\ell}) = f\ddl = f(T_1)=a_1.  
\end{equation}
Hence $\tY_n = a_1\neq0$, which contradicts \eqref{hh}.
This contradiction proves $f(\ct)=0$ a.s., which implies $F(\ct)=0$ and
$F(\ctn)=0$ \as{} for every $n$, completing the proof of (i).

Now consider (ii), with only $p_0$ and $p_r$ non-zero. (Since $\E\xi=1$, we
have $p_r=1-p_0=1/r$.) This is the case of full $r$-ary trees, and 
the random tree $\ctn$ has $(n-1)/r$ nodes of outdegree $r$ and $n-(n-1)/r$
leaves. Thus the choice $f(T)=\ett{T=\bullet}$, when $F(T)=n_\bullet(T)$ 
is the number of
leaves in $T$, yields $\Var F(\ctn)=0$ so $\gam^2=0$, see \refE{Eleaves}.

If $f$ is any functional with finite support such that $\gam^2=0$, we may
replace $f(T)$ by  $f(T)-f(\bullet)\ett{T=\bullet}$ without changing $\Var
F(\ctn)$, so we still have $\gam^2=0$. Hence we may assume $f(\bullet)=0$.
If we now consider  the case $\xi_i=0$, $i< n+m$, then by \eqref{g8},
$g(\xi_i,\dots,\xi_{i+m-1})=f(0)=0$ and thus \eqref{tg} yields
$\tg(\xi_i,\dots,\xi_{i+m-1})=-\mu-\mu/\gss$. Hence, \eqref{tyg} yields
$\tY_n=-n\mu(1+\gs\qww)$, and \eqref{hh} implies that this vanishes, 
and thus $\mu=0$. The rest of the proof is the same as for (i).
\end{proof}

\begin{proof}[Proof of \eqref{erc}] 
  This is a minor variation of arguments in Sections \ref{Sexp}--\ref{Svar},
  using the special simple structure of $f$. We omit some details.

If $T$ has degree sequence $\ddn$, then $n_r(T)=\sumin\ett{d_i=r}$. This
yields, arguing as for \eqref{Fd},
\begin{equation}
  n_r(\ctn)
\eqd \lrpar{\sum_{i=1}^{n} \ett{\xi_i=r}
\Bigm| S_n=n-1},
\end{equation}
jointly for all $r\ge0$.
This yields, arguing as in \eqref{efkn} and \eqref{vfnx},
\begin{equation}
  \label{ss1}
\E n_r(\ctn) 
= n\PP\bigpar{\xi_1=r\mid S_n=n-1}
=np_r\frac{\PP(S_{n-1}=n-1-r)}{\PP(S_{n}=n-1)}
\end{equation}
and, for any integers $r,s\ge0$,
\begin{equation*}
  \begin{split}
\E\bigpar{ n_r(\ctn) n_s(\ctn)}
&
=\gd_{rs}\E n_r(\ctn)
+ n(n-1)\PP\bigpar{\xi_1=r,\xi_2=s\mid S_n=n-1}
\\&
=\gd_{rs}\E n_r(\ctn)
+ n(n-1)p_rp_s\frac{\PP(S_{n-2}=n-1-r-s)}{\PP(S_{n}=n-1)}.
  \end{split}
\end{equation*}
Hence
{\multlinegap=0pt
\begin{multline}\label{scov}
  \Cov\bigpar{n_r(\ctn),n_s(\ctn)}
= 
\gd_{rs}\E n_r(\ctn)
-\frac1n \E n_r(\ctn)\E n_s(\ctn)
+ n(n-1)p_rp_s\times
\\ 
\biggpar{
\frac{\PP(S_{n-2}=n-1-r-s)}{\PP(S_{n}=n-1)}	
-\frac{\PP(S_{n-1}=n-1-r)}{\PP(S_{n}=n-1)}	
\frac{\PP(S_{n-1}=n-1-s)}{\PP(S_{n}=n-1)}	
}.
\end{multline}}%
For the mean, \eqref{ss1} and \eqref{llt} yield
\begin{equation}
  \label{spr}
\E n_r(\ctn)/n\to p_r, 
\end{equation}
\cf{} \eqref{ermu}. 
(The argument is simpler than in \refS{Sexp} since
we consider a fixed $r$.)

For the covariance, we argue as in the proof of \refL{Lv2} and consider
  \begin{multline}\label{solrs}
\PP\xpar{S_{n-2}=n-r-s-1}{\PP(S_n=n-1)} 
\\\shoveright{
-\PP\xpar{S_{n-1}=n-r-1}\PP\xpar{S_{n-1}=n-s-1}}
\\\shoveleft{\qquad
= \intpipi \tgf^{n-2}(t)e^{\ii(r+s-1)t}\dd t\cdot
\intpipi \tgf^{n}(u)e^{\ii u}\dd u
}
\\\shoveright{
-\intpipi \tgf^{n-1}(t)e^{\ii rt}\dd t\cdot 
\intpipi \tgf^{n-1}(u)e^{\ii su}\dd u}
\\
\shoveleft{\qquad
=
\frac{1}{8\pi^2}\intpipix\intpipix \tgf^{n-2}(t)e^{\ii t}
 \tgf^{n-2}(u)e^{\ii u} 
\bigpar{e^{\ii(r-1)t}\tgf(u)-\tgf(t)e^{\ii(r-1)u}}
\times}
\\
\bigpar{e^{\ii(s-1)t}\tgf(u)-\tgf(t)e^{\ii(s-1)u}}
\dd t\dd u.
  \end{multline}
We have
\begin{equation*}
  e^{\ii(r-1)t}\tgf(u)-\tgf(t)e^{\ii(r-1)u}
=\ii(r-1)(t-u) + O\bigpar{t^2+u^2},
\end{equation*}
and by a change of variables as in \eqref{sol00}--\eqref{sol000},
the final double integral in \eqref{solrs} is 
$\sim -(r-1)(s-1)/(2\pi\gs^4 n^2)$.
Hence \eqref{scov} yields, using also \eqref{snn} and \eqref{spr},
\begin{equation}
\frac{  \Cov\bigpar{n_r(\ctn),n_s(\ctn)}}{n}
\to
\gd_{rs}p_r
-p_rp_s
-\frac{(r-1)(s-1)}{\gss} p_rp_s
,
\end{equation}
showing \eqref{er2} and \eqref{erc}.
\end{proof}

\newcommand\AAP{\emph{Adv. Appl. Probab.} }
\newcommand\JAP{\emph{J. Appl. Probab.} }
\newcommand\JAMS{\emph{J. \AMS} }
\newcommand\MAMS{\emph{Memoirs \AMS} }
\newcommand\PAMS{\emph{Proc. \AMS} }
\newcommand\TAMS{\emph{Trans. \AMS} }
\newcommand\AnnMS{\emph{Ann. Math. Statist.} }
\newcommand\AnnPr{\emph{Ann. Probab.} }
\newcommand\CPC{\emph{Combin. Probab. Comput.} }
\newcommand\JMAA{\emph{J. Math. Anal. Appl.} }
\newcommand\RSA{\emph{Random Struct. Alg.} }
\newcommand\ZW{\emph{Z. Wahrsch. Verw. Gebiete} }
\newcommand\DMTCS{\jour{Discr. Math. Theor. Comput. Sci.} }

\newcommand\AMS{Amer. Math. Soc.}
\newcommand\Springer{Springer-Verlag}
\newcommand\Wiley{Wiley}

\newcommand\vol{\textbf}
\newcommand\jour{\emph}
\newcommand\book{\emph}
\newcommand\inbook{\emph}
\def\no#1#2,{\unskip#2, no. #1,} 
\newcommand\toappear{\unskip, to appear}

\newcommand\urlsvante{\url{http://www.math.uu.se/~svante/papers/}}
\newcommand\arxiv[1]{\url{arXiv:#1.}}
\newcommand\arXiv{\arxiv}

\def\nobibitem#1\par{}

\end{document}